\documentclass[11pt]{article}

\usepackage{amsfonts}
\usepackage{amsmath}
\usepackage{amssymb}
\usepackage{amsthm}
\usepackage{mathtools}
\usepackage{xcolor}
\usepackage{ bbold }
\usepackage{subfigure}

\theoremstyle{theorem}
\newtheorem{theorem}{Theorem}
\newtheorem{prop}[theorem]{Proposition}
\newtheorem{lemma}[theorem]{Lemma}
\newtheorem{corollary}[theorem]{Corollary}

\theoremstyle{definition}
\newtheorem{definition}[theorem]{Definition}

\DeclareMathOperator{\Deg}{Deg}
\DeclareMathOperator{\supp}{supp}

\makeatletter \renewcommand\@biblabel[1]{[#1]} \makeatother 

\title{A Model for Birdwatching and other \\ Chronological Sampling Activities}
\author{Jes\'us ~A. De Loera$^1$, Edgar Jaramillo-Rodriguez$^1$, \\ Deborah Oliveros$^2$, and Antonio J. Torres$^2$}
\date{%
    $^1$Department of Mathematics, University of California, Davis\\%
    $^2$ Instituto de Matem\'aticas, Universidad Nacional Aut\'onoma de M\'exico\\[2ex]%
    \today
}
\begin{document}

%\markright{A Model via Random Interval Graphs}
\maketitle

\begin{abstract} In many real life situations one has $m$ types of random events happening in chronological order within a time interval and one wishes to predict various milestones about these events or their subsets. 
An example is birdwatching. Suppose we can observe up to $m$ different types of birds during a season. At any moment a bird of type $i$ is observed with some probability. There are many natural questions a birdwatcher may have: how many observations should one expect to perform before recording all types of birds? Is there a time interval where the researcher is most likely to observe all species? Or, what is the likelihood that several species of birds will be observed at overlapping time intervals?  Our paper answers these questions using a new model based on random interval graphs. This model is a natural follow up to the famous coupon collector's problem.
\end{abstract}

%%%%%%%%%%%%%%%%%%%
%%%%%%%%%%%%%%%%%%%
%%%%----Introduction
%%%%%%%%%%%%%%%%%%%
%%%%%%%%%%%%%%%%%%%

\section{Introduction.}\label{intro}
Suppose you are an avid birdwatcher and you are interested in the migratory patterns of different birds passing through your area this winter. Each day you go out to your backyard and keep an eye on the skies; once you see a bird you make a note of the species, day, and time you observed it. You know from prior knowledge that there are $m$ different species of birds that pass over your home every year and you would love to observe at least one representative of each species. Naturally, you begin to wonder: {\em after $n$ observations, how likely is it that I have seen every type of bird?} If we only care that all $m$ types of birds are observed at least once after $n$ observations, we recognize this situation as an example of the famous \emph{coupon collector's problem} (for a comprehensive review of this problem see \cite{Coupon} and references therein). In this old problem a person is trying to collect $m$ types of objects, the coupons, labeled $1,2,\dots ,m$. The coupons arrive one by one as an ordered sequence $X_1,X_2, \ldots$ of independent identically distributed (i.i.d.) random variables taking values in $[m] = \{1,\ldots, m\}$. 

But a professional birdwatcher is also interested in more nuanced information than the coupon collector. To properly understand interspecies interactions, one not only hopes to observe every bird, but also needs to know which species passed through the area at the same time(s). For example, the birdwatcher might also  ask:

\begin{itemize}
\item \emph{What are the chances that the visits of $k$ types of birds do not overlap at all?}

\item \emph{What are the chances that a pair of birds is present on the same time interval?}

\item \emph{What are the chances of one bird type overlapping in time with $k$ others?}

\item \emph{What are the chances that all the bird types overlap in a time interval?}
\end{itemize}

We note that very similar situations, where scientists collect or sample time-stamped data that comes in $m$ types or 
classes and wish to predict overlaps, appear in applications as diverse as archeology, genetics, job scheduling, and paleontology \cite{GOLUMBIC,Fishburn85,pippenger,paleobook}.
%Mathematically, we need a good model that captures the time-overlap questions above. 
The purpose of this paper is to present a new \emph{random graph model} to answer the four time-overlap questions above. 

Our model is very general, but to avoid unnecessary formalism and technicalities, we present clear answers 
 in some natural special cases that directly generalize the coupon collector problem. For the special cases we analyze, the only tools we use are a combination of elementary probability and combinatorial geometry.

\subsection{Establishing a general random interval graph model.}

In order to answer any of the questions above we need to deal with one key problem: how do we estimate which time(s) each species of bird might be present from a finite number of observations? To do so, we will make some 
modeling choices which we outline below. 

The first modeling choice is that our observations are samples from a stochastic process indexed by a real interval $[0,T]$ and taking values in $[m]$. We recall the definition of a stochastic process for the reader (see {\cite{StochProcess}): Let $I$ be a set and let $(\Omega, \mathcal{F}, P)$ be a probability space. Suppose that for each $\alpha \in I$, there is a random variable $Y_\alpha : \Omega \to S \subset \mathbb{R}$ defined on $(\Omega, \mathcal{F}, P)$. Then the function $Y : I \times \Omega \to S$ defined by $Y(\alpha, \omega) = Y_\alpha(\omega)$ is called a \emph{stochastic process} with \emph{indexing set} $I$ and \emph{state space} $S$, and is written $Y = \{Y_\alpha : \alpha \in I\}$. 
When we conduct an observation at some time $t_0 \in [0,T]$, we are taking a sample of the random variable $Y_{t_0}$.

For each $i\in [m]$, the probabilities that $Y_t=i$ give us a function from $[0,T] \to [0,1]$, which we
call the \emph{rate function} of $Y$ corresponding to $i$; the name is inspired by the language of Poisson point processes where the density of points in an interval is determined by a \emph{rate} parameter (see \cite{Ross_Stoch}). 

\begin{definition}[Rate function]
Let $Y = \{Y_t: t \in[0,T]\}$ be a stochastic process with indexing set $I = [0,T]$ and state space $S = [m]$. 
The \emph{rate function} corresponding to label 
$i\in S$ in this process is the function $f_i : I \to [0,1]$ given by $$f_i(t)=P(Y_t =i)= P(\{\omega: Y(t,\omega)=i\}).$$  
\end{definition}

Figure \ref{fig:2examples} gives two examples of the rate functions of some hypothetical stochastic processes (we will clarify the meaning of stationary and non-stationary later in this section when we discuss a special case of our model). Observe that at a fixed time $t_0$, the values $f_i(t_0)$ sum to 1 and thus determine the probability density function of $Y_{t_0}$. Therefore, the rate functions describe the change of the probability density functions of the variables $Y_t$ with respect to the indexing variable $t$.

Next, note that the set of times where species $i$ might be present is exactly the \emph{support} of the rate function $f_i$. Recall, the support of a function is the subset of its domain for which the function is non-zero, in our case this will be a portion of $[0,T]$. Therefore, \emph{our key problem is to estimate the support of the rate functions from finitely many samples}. 

\begin{figure}[h]
\centering     
\subfigure[Stationary]{\label{fig:stat_timeline}\includegraphics[width=65mm]{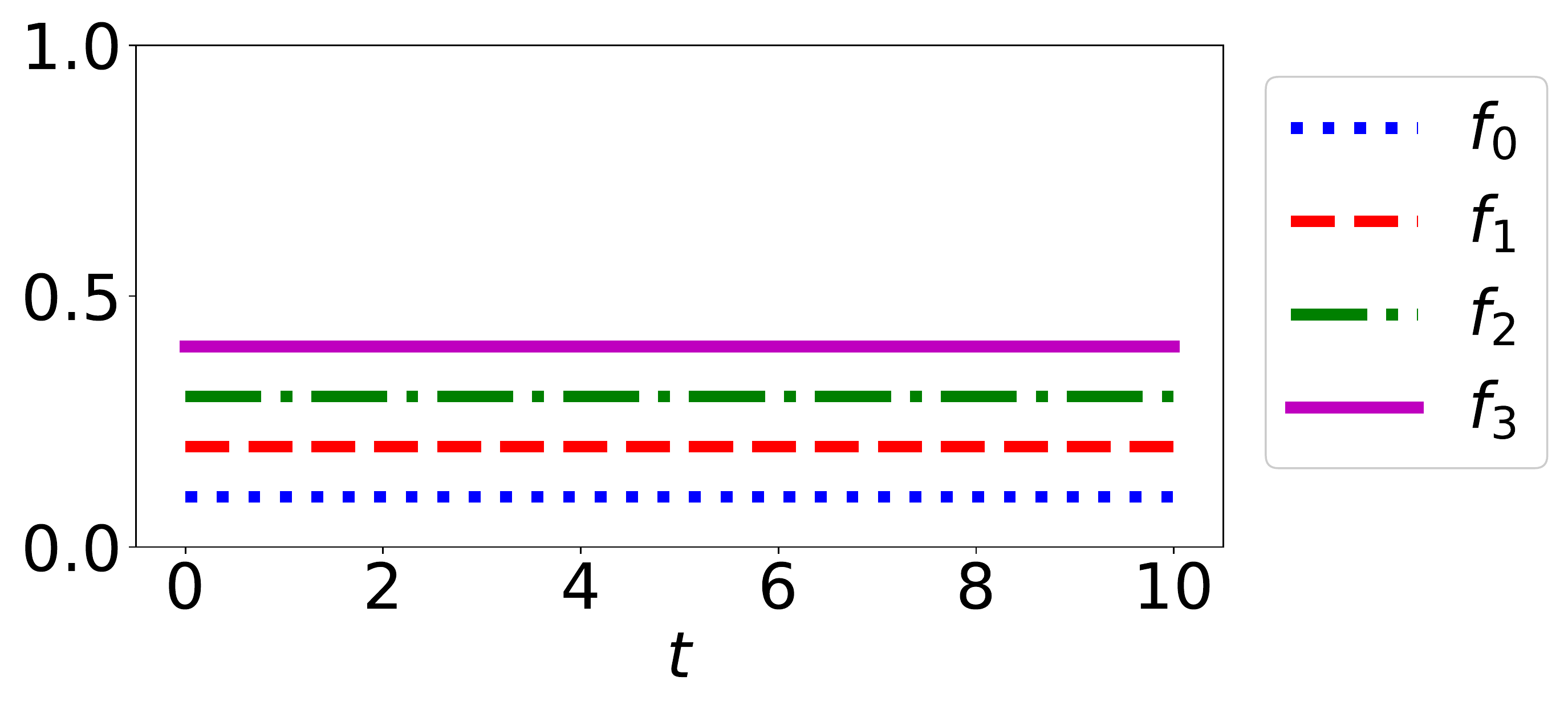}}
\subfigure[Non-Stationary]{\label{fig:timeline}\includegraphics[width=59mm]{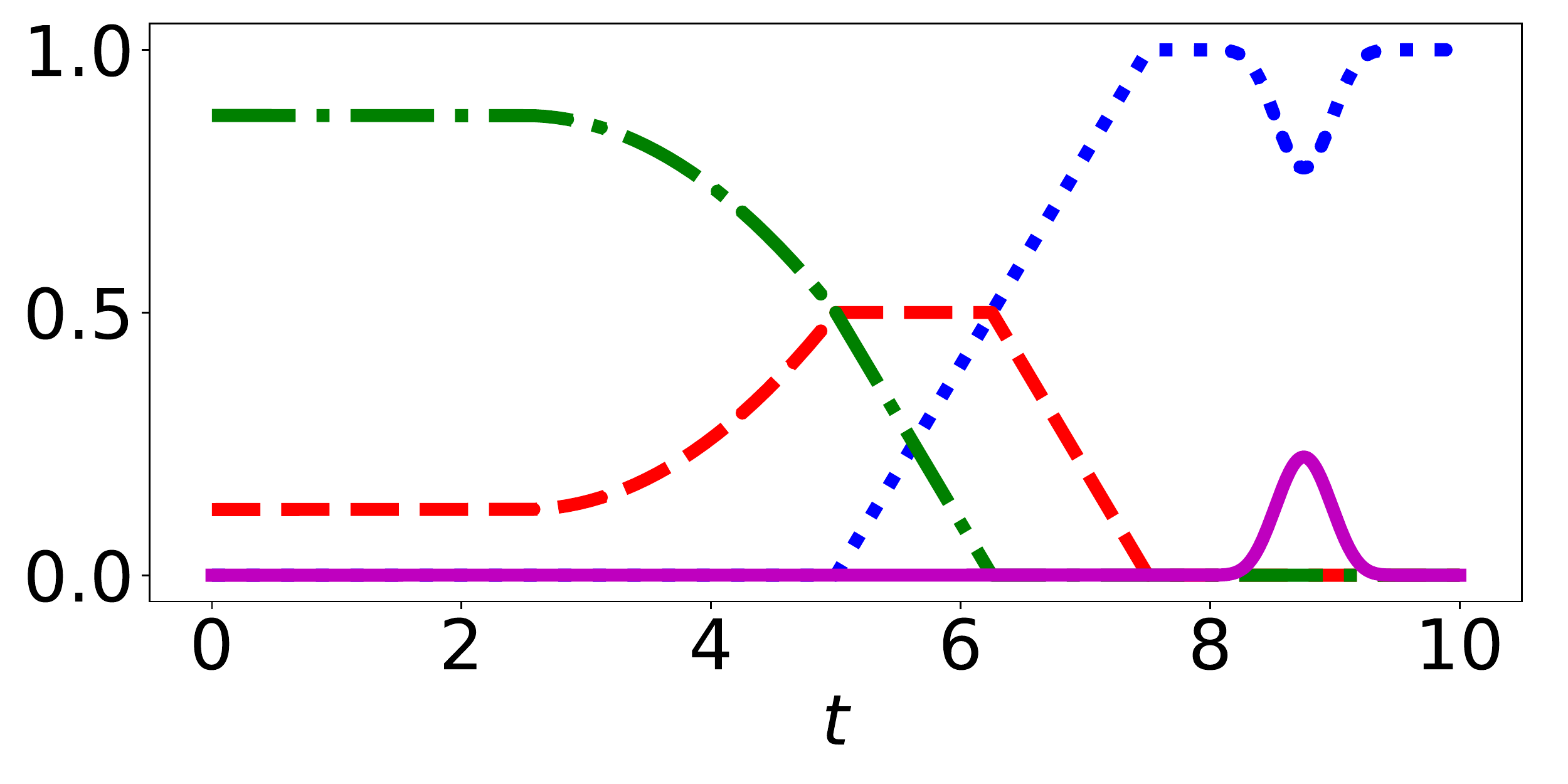}}
\caption{Two examples of hypothetical rate functions.}
{\label{fig:2examples}}
\end{figure}

%\textcolor{red}
We note that the stochastic process $Y$ is defined to take values in $[m]$ due to a modeling choice on our part. Alternatively, one could have $Y$ take values in the power set $2^{[m]}$, so as to allow for multiple species of birds to be observed at the same time. However, choosing $[m]$ rather than $2^{[m]}$ simplifies some calculations and, moreover, is quite reasonable. Rather than registering ``three birds at 6 o'clock," our birdwatcher can instead register three sightings: one bird at 6:00:00, a second at 6:00:01, and a third a 6:00:02, for example. 

This brings us to our next modeling choice: all the rate functions $f_i$ have connected support for each $i \in [m]$. This is reasonable for our motivation; after all, a bird species first seen on a Monday and last seen on a Friday is not likely to suddenly be out of town on Wednesday. The main benefit of this assumption is that now the support of the rate function $f_i$, $\supp(f_i)$, is a sub-interval of $[0,T]$. This fact provides a natural way of approximating the support of $f_i$: given a sequence of observations $Y_{t_1}, Y_{t_2} , \ldots, Y_{t_n}$ with $0 \leq t_1 < t_2 < \ldots < t_n \leq T$, let $I_n(i)$ denote the sub-interval of $[0, T]$ whose endpoints are the first and last times $t_k$ for which $Y_{t_k} = i$. Note that it is possible for $I_n(i)$ to be empty or a singleton. It follows that $I_n(i) \subset \supp(f_i)$ so we can use it to approximate $\supp(f_i)$. We call the interval $I_n(i)$ the \emph{empirical support} of $f_i$, as it is an approximation of $\supp(f_i)$ taken from a random sample. 

In summary, our model is actually quite simple: given a sequence of observations $Y_{t_1}, Y_{t_2} , \ldots, Y_{t_n}$ we construct $m$ random intervals $I_n(1), \ldots, I_n(m)$ whose endpoints are the first and last times we see its corresponding species. These intervals, known as the \emph{empirical supports}, are approximations of the supports of the rate functions, $f_i$, and satisfy $\supp(f_i) \supset I_n(i)$. 

The four birdwatching questions above may be expressed in terms of the empirical supports as follows:

\begin{itemize}
\item \emph{What are the chances that none of the empirical supports $I_n(i)$ intersect?}

\item \emph{What are the chances that a particular pair of empirical supports $I_n(i)$ and $I_n(j)$ intersect?}

\item \emph{What are the chances that one empirical support, $I_n(i)$, intersects with $k$-many others?}

\item \emph{What are the chances that the collection of empirical supports has a non-empty intersection?}
\end{itemize}

To make these questions even easier to analyze, we will present a combinatorial object: an \emph{interval graph} that records the intersections of the intervals $I_n(i)$ in its edge set. 

\begin{definition} Given a finite collection of $m$ intervals on the real line, its corresponding interval graph, $G(V,E)$, is the simple graph with $m$ 
vertices, each associated to an interval, such that an edge $\{i,j\}$ is in $E$ if and only if the associated intervals have a nonempty intersection, i.e., they overlap. 
\end{definition}

Figure \ref{fig:nerve_example} demonstrates how we construct the desired interval graph from some observations. Figure \ref{fig:data} shows a sequence of $n=11$ points on the real line, which corresponds to the indexing set $I$ of our random process $Y$. Above each point we have a label, representing a sample from $Y$ at that time. Displayed above the data are the empirical supports $I_n(i)$ for each $i \in [m] = [4]$. Finally, Figure \ref{fig:int_graph} shows the interval graph constructed from these four intervals where each vertex is labeled with the interval it corresponds to. In this example there are no times shared by the species $\{1,2\}$ and the species $\{4\}$, so there are no edges drawn between those nodes. 

We emphasize that the interval graph constructed in this way will contain up to $m$-many vertices, but may contain fewer if some of the intervals $I_n(i)$ are empty, i.e., if we never see species $i$ in our observations.

\begin{figure}[h]
\centering     
\subfigure[Labeled observations and induced intervals]{\label{fig:data}\includegraphics[width=55mm]{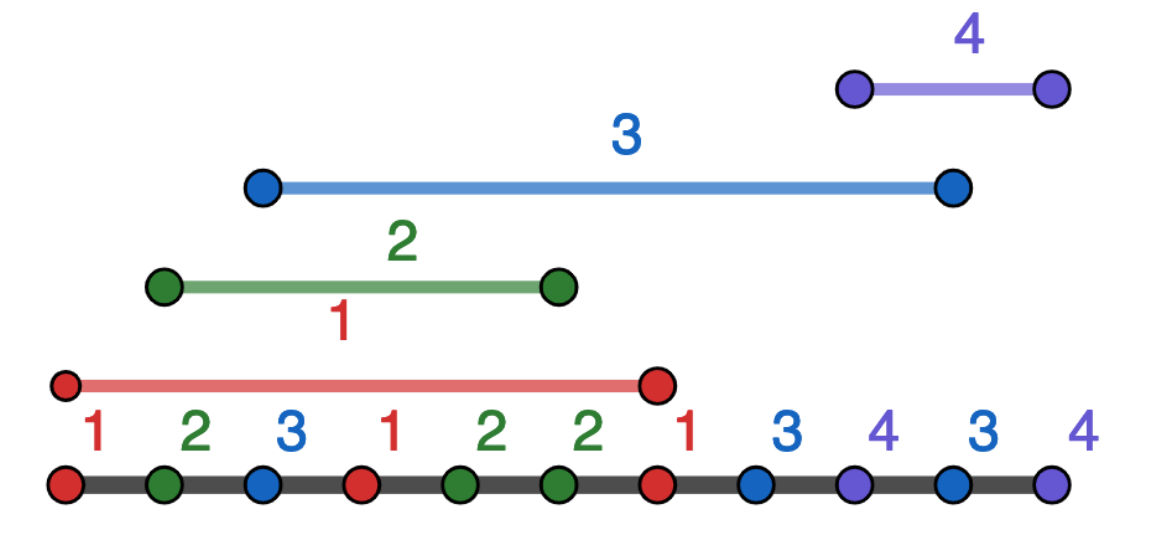}}
\subfigure[Interval Graph]{\label{fig:int_graph}\includegraphics[width=30mm]{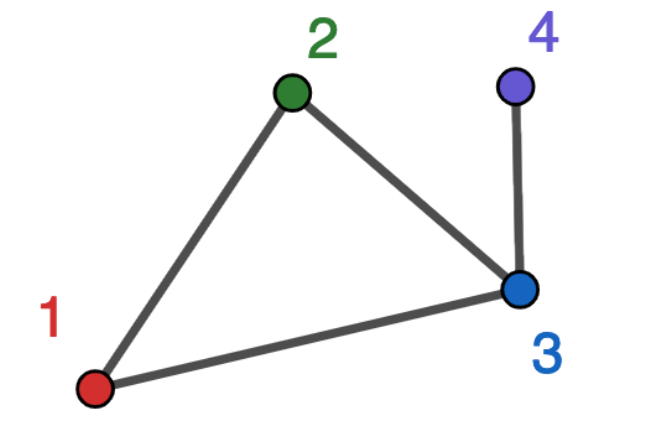}}
\subfigure[Nerve Complex]{\label{fig:nerve}\includegraphics[width=30mm]{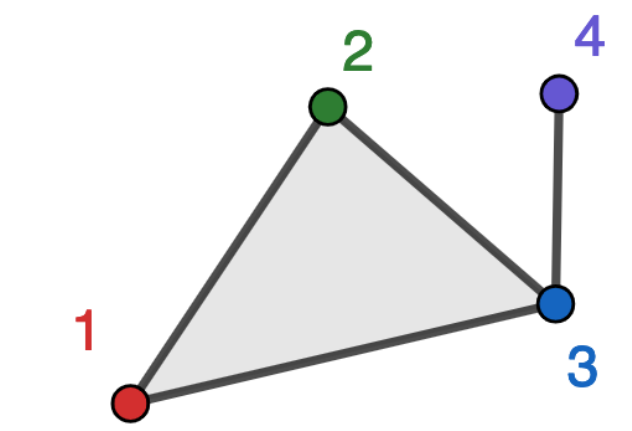}}
\caption{Example observations with their corresponding graph and nerve.}
\label{fig:nerve_example}
\end{figure}

Although the interval graph $G(V,E)$ is constructed using only pairwise intersections, we can further encode all $k$-wise intersections for $k = 2, \ldots, m$ in a \emph{simplicial complex}, which is a way to construct a topological space by gluing \emph{simplices} (generalizations of triangles, tetrahedra, etc). A  simplicial complex $K$ must satisfy the two requirements that every face of a simplex in $K$ is also in $K$ and that the non-empty intersection of any two simplices in $K$ is a face of both.
(for an introduction to basic topology and simplicial complexes see \cite{ghrist2014elementary,Hatcher}). The construction we need is known as the \emph{nerve complex} (see \cite{kozlovbook}, \cite{tancer}, \cite[p.~197]{matousek2002lectures} and \cite[p.~31]{ghrist2014elementary}).

\begin{definition} 
Let $\mathcal{F} = \{F_1,\ldots,F_m\}$ be a family of convex sets in $\mathbb{R}^d$. The \emph{nerve complex}
$\mathcal{N}(\mathcal{F})$ is the abstract simplicial complex whose $k$-facets are the $(k+1)$-subsets $I \subset [m]$ such that $\bigcap_{i\in I} F_i \neq \emptyset$.
\end{definition}

Figure \ref{fig:nerve} shows the nerve complex constructed from the intervals $I_n(i)$ in Figure \ref{fig:data}. Note the presence of a 2-simplex (triangle) with vertices $\{1, 2, 3\}$ because the corresponding intervals mutually intersect.

%\textcolor{red}
By construction, the interval graph $G$ is exactly the 1-skeleton of the nerve complex $\mathcal N$ generated by the intervals. In fact, because our intervals lie in a 1-dimensional space, $\mathcal N$ is completely determined by $G$. To see this, suppose we have a collection of intervals $(x_1,y_1), \ldots, (x_k,y_k)$ such that all intervals intersect pairwise. It follows that $y_i \geq x_j$ for all $i,j \in [k]$, and so $(\max \{x_1, \ldots,x_k\}, \min\{y_1, \ldots, y_k \})$ $\subseteq \cap_{i=1}^k (x_i,y_i)$. Hence the whole collection has non-empty intersection (this is a special case of Helly's theorem \cite{Barvinok}, which is necessary in higher dimensional investigations). Thus, the $k$-dimensional faces of the nerve complex are precisely $k$-cliques of the interval graph. 

Therefore, going forward we will refer to the nerve complex $\mathcal N$ and the graph $G$ interchangeably depending on the context, but the reader should understand that these are fundamentally the same object as long as the family of convex sets $\mathcal F$ lies in a 1-dimensional space. We stress that in higher dimensions the intersection graph of convex sets \emph{does not} determine the nerve complex (we demonstrate this by an example in the Conclusion).

We can now present our random interval graph model in its entirety: 

 \begin{definition}[The Random Interval Graph Model]
We let $Y = \{ Y_t : t\in [0,T]\}$ be a stochastic process as above and let $\mathcal{P}=\{ t_1,t_2,...,t_n\}$ be a set of $n$ distinct observation times or sample points in $[0,T]$ with $t_1 < t_2 < \ldots < t_n$.  Then let $Y = (Y_1, Y_2, \ldots, Y_n)$ be a random vector whose components $Y_i$ are samples from $Y$ where $Y_i = Y_{t_i}$, so each $Y_i$ takes values $\{ 1, \ldots, m\}$. For each label $i$ we define the (possibly empty) interval $I_n(i)$ as the convex hull of the points $t_j$ for which $Y_j =i$, i.e., the interval defined by points colored $i$. Explicitly $I_n(i) = \text{Conv}(\{t_j \in \mathcal{P} : Y_j = i\})$, and we refer to $I_n(i)$ as the \emph{empirical support} of label $i$. 
Furthermore, because it comes from the $n$ observations or samples, we call the nerve complex, $\mathcal N(\{I_n(i): i =1, \ldots m \})$, the \emph{empirical nerve} of $Y$ and denote it $\mathcal N_n(Y)$. 
\end{definition}

Under this random interval graph model our four questions can be rephrased in terms of the random graph $\mathcal N_n(Y)$:

\begin{itemize}
\item \emph{What is the likelihood that $\mathcal N_n(Y)$ has no edges?}

\item \emph{What is the likelihood that a particular edge is present in $\mathcal N_n(Y)$?}

\item \emph{What is the likelihood of having a vertex of degree at least $k$ in $\mathcal N_n(Y)$?}

\item \emph{What is the likelihood that $\mathcal N_n(Y)$ is the complete graph $K_m$?}
\end{itemize}

Our original questions have become questions about random graphs!

\subsection{The special case this paper analyzes.} We presented a very general model because it best captures the nuances and subtleties of our motivating problem. However, without additional assumptions on the distribution $Y$, the prevalence of pathological cases makes answering the motivating questions above become very technical and unintuitive.
Therefore,  our analysis will focus on a special case of this problem where we make two additional assumptions on $Y$ so that our analysis only requires basic combinatorial probability.

The first assumption we make is that our observations $Y_{t_1}, Y_{t_2}, \ldots, Y_{t_n}$ are mutually independent random variables. Note, we do not claim that all pairs of random variables $Y_s, Y_t$ for $s,t \in [0,T]$ are independent. We only claim this holds for all $s,t \in \{t_1, t_2, \ldots, t_n\}$.
The second assumption we make is that the rate functions $f_i$ be constant throughout the interval $[0,T]$. In this case, there exist constants $p_1, p_2, \ldots, p_m \geq 0$ such that $\sum_{i=1}^m p_i = 1$ and $f_i(t) = p_i$ for all $t\in [0,T]$ and all $i \in [m]$. We call the special case of our model where both of these assumptions are satisfied the \emph{stationary case} and all other cases as \emph{non-stationary}. Figure \ref{fig:2examples} shows examples 
of a stationary case, \ref{fig:stat_timeline}, and a non-stationary case, \ref{fig:timeline}. We will also refer to the \emph{uniform case}, which is the extra-special situation where $p_i=\frac{1}{m}$ for all $i\in [m]$. Note Figure \ref{fig:stat_timeline} is stationary but not uniform.

Of course, the stationary case is less realistic and applicable in many situations. For example, it is not unreasonable to suppose that the presence of a dove at 10 o'clock should influence the presence of another at 10:01, or that the presence of doves might fluctuate according to the season and time of day. However, the stationary case is still rich in content and, importantly, simplifies things so that this analysis requires only college-level tools from probability and combinatorics. Moreover, as we discuss below, the stationary case has a strong connection to the famed coupon collector problem and is of interest as a novel method for generating random interval graphs. 

The stationary case assumptions directly lead to two important consequences that greatly simplify our analysis. The first is that now the random variables $Y_{t_1} ,\ldots, Y_{t_n}$ are independent and identically distributed (i.i.d.) such that $P(Y_{t_k} = i) =p_i >0$. Note that this is true for any set of distinct observation times $\mathcal P = \{t_1, \ldots, t_n\}$. The second consequence simplifies things further still: though the points $\mathcal{P}$ corresponding to our sampling times have thus far been treated as arbitrary, one can assume without loss of generality that $\mathcal{P} =[n]= \{1,2,\ldots, n\}$ since all sets of $n$ points in $\mathbb{R}$ are combinatorially equivalent, as explained in the following lemma.

\begin{lemma}
\label{stat_lemma}
Let $\mathcal{P} = \{x_1, \ldots, x_n \}$ and $\mathcal{P}' = \{x_1', \ldots, x_n' \}$ be two sets of $n$ distinct points in $\mathbb{R}$ with $x_1 < \ldots < x_n$ and $x_1' < \ldots < x_n'$. Let $Y = (Y_1, \ldots, Y_n)$ and $Y' = (Y_1', \ldots, Y_n')$ be i.i.d. random vectors whose components are i.i.d. random variables taking values in $[m]$ with $P(Y_j = i) = p_i > 0$ and $P(Y^{\prime}_j = i) = p_i > 0$. Then for any abstract simplicial complex $\mathcal{K}$ we have that $P(\mathcal{N}_n(\mathcal P, Y) = \mathcal{K}) = P(\mathcal{N}_n(\mathcal P', Y') = \mathcal{K})$.
\end{lemma}

\begin{proof}
Let $c_1,c_2,\ldots, c_n$ be an arbitrary sequence of labels, so $c_i \in [m]$ for each $i$. Because $Y,Y'$ are i.i.d. we have that  $P(\cap_{i=1}^n \{Y_i =c_i)\}) = P(\cap_{i=1}^n (\{Y_i' =c_i\}).$
Therefore if both sequences of colors $Y_i = Y_i' = c_i$ have the same order for all $i =1,\ldots, n$, then it is sufficient to show that the two empirical nerves are the same. Consider two empirical supports $I_n(j)$ and $I_n(k)$ of labels $j,k$, and observe that if they do (do not) intersect on $Y_i$, then the two empirical supports $I^{\prime}_n(j)$ and $I^{\prime}_n(k)$ of labels $j,k$ do (do not) intersect, then the two corresponding empirical nerves do (do not) contain the edge $\{j,k\}$. This implies that the two nerves have the same edge set. Furthermore, as we observed before, due to Helly's theorem in the line the empirical nerve is completely determined by its 1-skeleton. Then both empirical nerves are the same.
\end{proof}

We now summarize the key assumptions of our model in the stationary case. 

{\bf Key assumptions for our analysis:} \emph{ In all results that follow
let $Y = (Y_1, \ldots, Y_n)$ be a random vector whose components are i.i.d. random variables such that $P(Y_j = i) = p_i >0$ for all $i \in [m]$. As a consequence the support functions of the underlying stochastic process are constant and each has support on the entire domain.
We denote by $\mathcal{N}_n = \mathcal{N}_n([n], Y)$ the empirical nerve of the random coloring induced by $Y$. We also denote the graph or 1-skeleton of  $\mathcal{N}_n$ by the same symbol. When we refer to the uniform case this means the special situation when $p_i=\frac{1}{m}$ for all $i=1,\dots, m$.}
 
\subsection{Context and prior work.}

We want to make a few comments to put our work in context and mention prior work:

The famous coupon collector problem that inspired us dates back to 1708 when it first appeared in De Moivre's \textit{De Mensura Sortis (On the Measurement of Chance)} \cite{Coupon}. The answer for the coupon collector problem depends on the assumptions we make about the distributions of the $X_i$. Euler and Laplace proved several results when the coupons are equally likely, that is when $P(X_i = k) = \frac{1}{m}$ for every $k\in [m]$. The problem lay dormant until 1954 when H. Von Schelling obtained the expected waiting time when the coupons are not equally likely \cite{Schelling}. More recently, Flajolet et. al. introduced a unified framework relating the coupon collector problem to many other random allocation processes \cite{FLAJOLET}. We note that the stationary case of our model has the same assumptions as this famous problem: an observer receives a sequence of i.i.d. random variables taking values in $[m]$. In the language of our model, the coupon collector problem could be posed as, \emph{What is the likelihood that the nerve} $\mathcal{N}_n(Y)$ \emph{will contain exactly m vertices?} Thus, we can consider this model a generalization of the coupon collector problem which seeks to answer more nuanced questions about the arrival of different coupons.

Interval graphs have been studied extensively due to their wide applicability  in areas as diverse as archeology, 
genetics, job scheduling, and paleontology \cite{GOLUMBIC,Fishburn85,pippenger,paleobook}.  These graphs 
have the power to model the overlap of spacial or chronological events and allow for some inference of structure. 
There are also a number of nice characterizations of interval graphs that have  been obtained \cite{Lekkeikerker,fulkersongross,gilmore_hoffman,hanlon82}.
For example, a graph $G$ is an interval graph if and only if the maximal cliques of $G$ can be linearly ordered in such a 
way that, for every vertex $x$ of $G$, the maximal cliques containing $x$ occur consecutively in the list. Another remarkable fact of interval graphs is that they are \emph{perfect} and thus the weighted clique and coloring problems are polynomial time solvable \cite{GOLUMBIC}.  Nevertheless, sometimes it may not be immediately clear whether a graph is an interval graph or not. For example, of the graphs in Figure \ref{fig:graph_example} only \ref{fig:graph1} is an interval graph. 

\begin{figure}[h]
\centering     %%% not \center
\subfigure[]{\label{fig:graph1}\includegraphics[width=42mm]{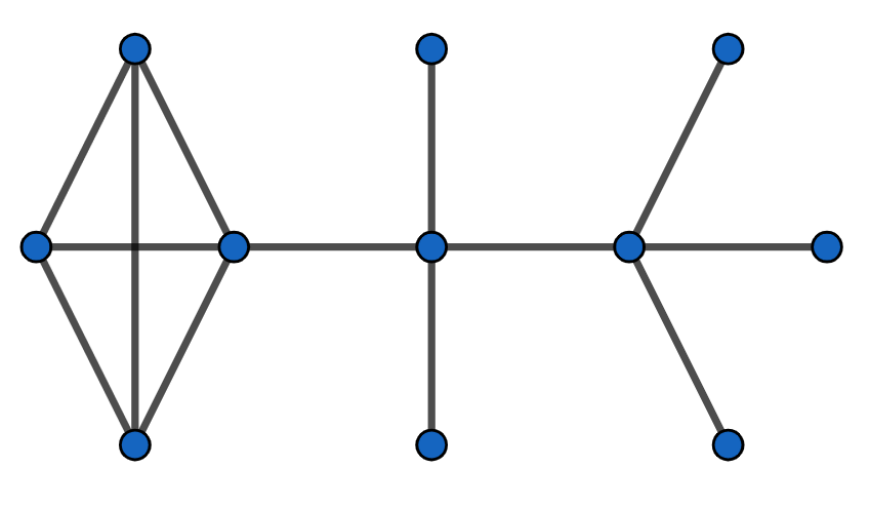}}
\subfigure[]{\label{fig:graph2}\includegraphics[width=25mm]{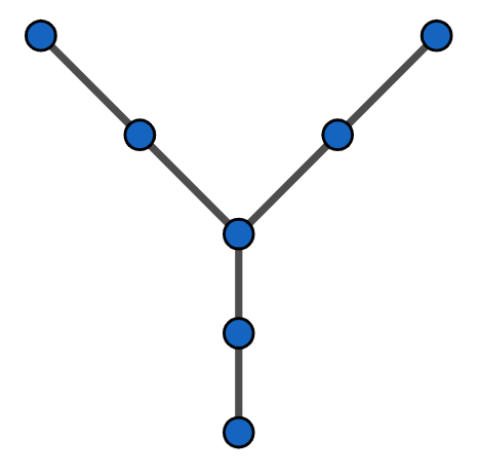}}
\subfigure[]{\label{fig:graph3}\includegraphics[width=25mm]{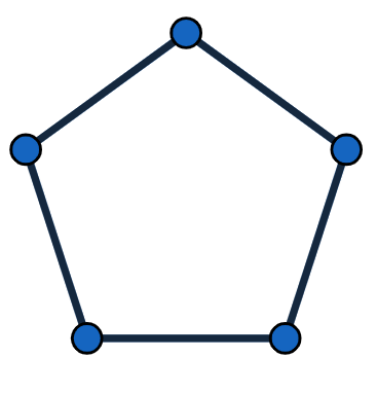}}
\caption{It is not obvious which of these graphs are interval.}
\label{fig:graph_example}
\end{figure}

The most popular model for generating random graphs is the Erd\H{os}-Renyi model \cite{erdos-renyi}, but it is insufficient for studying random interval graphs. The reason is that, as pointed out in \cite{cohenetal1979probability}, an Erd\H{os}-Renyi graph is almost certainly \emph{not} an interval graph as the number of vertices goes to infinity.

Several other authors have studied various models for generating random \emph{interval graphs} (see \cite{diaconis2013interval, Scheinermanoriginal, Scheinerman2, JusticzScheinermanWinkler, iliopoulos, pippenger} and the many references therein). Perhaps most famously Scheinerman introduced \cite{Scheinermanoriginal,Scheinerman2}, and others 
investigated \cite{diaconis2013interval,JusticzScheinermanWinkler,iliopoulos}, a method of generating random 
interval graphs with a fixed number of intervals $m$:  the extremes of the intervals  $\{(x_1, y_1),\dots, (x_m, y_m)\}$ 
are $2m$ points chosen independently from some fixed continuous probability distribution on the real line. Each pair 
$(x_i, y_i)$ determines a random interval. This is a very natural simple random process, but it is different from our 
random process (see the Appendix).

%\textcolor{red}
%There are, of course, many models for generating random graphs, starting from the famous Erd\H{os}-Renyi model \cite{erdos-renyi}. However, Cohen, Koml\'os, and Mueller \cite{cohenetal1979probability} showed that the probability that an Erd\H{o}s-Renyi random graph with $n$ vertices forms an interval graph goes to zero as $n$ goes to infinity. 

%There are many models for generating random graphs. The most famous is the Erd\H{os}-Renyi model \cite{erdos-renyi}, but it is insufficient to study random interval graphs. The reason is that Cohen, Koml\'os, and Mueller \cite{cohenetal1979probability} showed that the probability that an Erd\H{os}-Renyi graph is an interval graph goes to 0 as the number of vertices goes to infinity.

%There are many models for generating random graphs. The most famous is the Erd\H{os}-Renyi model \cite{erdos-renyi}. However,  Cohen, Koml\'os, and Mueller \cite{cohenetal1979probability} showed that an Erd\H{os}-Renyi graph is almost certainly \emph{not} an interval graph as the number of vertices goes to infinity.

%The most famous model for generating random graphs is the Erd\H{os}-Renyi model \cite{erdos-renyi}. However, an Erd\H{os}-Renyi graph is almost certainly \emph{not} an interval graph as the number of vertices goes to infinity \cite{cohenetal1979probability}. Therefore, the Erd\H{os}-Renyi model is insufficient for studying random interval graphs.

We noted earlier that because our intervals lie in a 1-dimensional space, the nerve complex  is completely determined by the interval graph because the $k$-facets of the nerve complex are exactly the $k$-cliques of the interval graph. In other words, the nerve complex is precisely the \emph{clique complex} of the interval graph. We also remark that the complement graph of the interval graph $G$ is the graph $H$ of non-overlapping intervals. The graph $H$ is in fact a partially ordered set, called the \emph{interval order} where one interval is less than the other if the first one is completely to the left of the second one. We can associate to each \emph{independent set} of $k$ non-intersecting intervals, a $(k-1)$-dimensional simplex, this yields a simplicial complex, the \emph{independence complex} of the corresponding interval order graph $H$. Observe that this independence complex is the same as the nerve $\mathcal N$  we just defined above. This is all well-known since the independence complex  of any graph equals the clique complex of its complement graph, and vice versa (see Chapter 9 in \cite{kozlovbook}).

\subsection{Outline of our contributions.}

In this paper we answer the four birdwatching questions using the random interval graphs and complexes generated by the stochastic process described above. Here are our results section by section:

Section \ref{sec:expectation} presents various results about the expected structure of the random interval graph $\mathcal{N}_n$, including the expected number of edges and the likelihood that the graph has an empty edge set. 

Section \ref{sec:cliques} presents results regarding the distribution of maximum degree and clique number of the graph $\mathcal{N}_n$, and our results show that the random interval graph asymptotically approximates the complete graph, $K_m$, as the number of samples $n$ grows large. This means the nerve complex is asymptotically an $(m-1)$-dimensional simplex. From the results of Section \ref{sec:cliques} one can see that as we sample more and more bird observations it becomes increasingly unlikely that we see any graph other than $K_m$. We investigate the number of samples needed to find $K_m$ with high probability.

Section \ref{conclusiones} closes the paper outlining three natural open questions. We also include an Appendix that contains computer experiments to evaluate the quality of various bounds proved throughout the paper and to show our model is different from earlier models of random interval graphs.

%%%%%%%%%%%%%%%%%%%
%%%%%%%%%%%%%%%%%%%
%%%%----Behaviour in Expectation
%%%%%%%%%%%%%%%%%%%
%%%%%%%%%%%%%%%%%%%
\section{Random Interval Graphs and Behavior in Expectation.} \label{sec:expectation}

In this section we explore what type of nerve complexes one might expect to find for a fixed number of observations $n$ when the likelihood of observing each label $i$ is a constant $p_i>0$. 

\begin{prop}\label{Null_small_prop}  Under the key assumptions in Section \ref{intro}, the probability that the random graph $\mathcal{N}_n$ is the empty graph with $0\leq k \leq m$ vertices but no edges, $K_k^c$, is given by
$$P(\mathcal{N}_n=K_k^c)\geq p_{*}^n  k! \binom{m}{k}\binom{n-1}{k-1},$$ 
where $p_{*}=\min\{p_1,p_2,$ $...,p_m\}$. 
Moreover, if  $p_i = \frac{1}{m}$ for all $i \in [m]$, then
$$P(\mathcal{N}_n=K_k^c)= \frac{k!}{m^n} \binom{m}{k}\binom{n-1}{k-1}.$$ 
\end{prop}

\begin{proof} 
Note that for $\mathcal{N}_n$ to form a disjoint collection of $k$ points, the intervals  induced by the coloring must also be disjoint. This occurs if and only if all points of the same color are grouped together. Given $k$ fixed colors it is well known that the disjoint groupings are counted by the number of compositions of $n$ into $k$ parts, $\binom{n-1}{k-1}$. Each composition occurs with probability at least $p_{*}^n$. Finally, considering the $\binom{m}{k}$ different ways to choose these $k$ colors and the $k!$ ways to order them, we have that, $$P(\mathcal{N}_n=K_k^c)\geq p_{*}^n k! \binom{m}{k} \binom{n-1}{k-1}.$$ 

The last statement  follows the same idea  but here every $k-$coloring of the $n$ points  happens with probability $\frac{1}{m}$. 
\end{proof}

Next we bound the probability that a particular edge is present in the random interval graph.

\begin{theorem}\label{ijedges} Under the key assumptions in Section \ref{intro} and
%Let $Y = (Y_1, \ldots, Y_n)$ be a random vector whose components are i.i.d. random variables such that $P(Y_j = i) = p_i >0$ for all $i \in [m]$. Let $\mathcal{N}_n = \mathcal{N}_n([n], Y)$ denote the empirical nerve of the random coloring induced by $Y$. 
for any pair $\{i,j\}$, $1\leq i < j \leq m$,  the probability of event $A_{ij} =\{\{i,j\} \in \mathcal{N}_n \}$, i.e., that the edge $\{i,j\}$ is present in the graph $\mathcal{N}_n$ equals
%$$ P(A_{ij}) = 1- (1-p_{ij})^n -\sum_{k=1}^n  \binom{n}{k}\bigg[ \bigg( 2 \sum_{r=1}^{k-1} p_i^r p_j^{k-r} \bigg) +p_i^k +p_j^k \bigg] (1-p_{ij})^{n-k},$$
%where $p_{ij} := p_i +p_j$.
$$ P(A_{ij}) = 1-q_{ij}^n -\sum_{k=1}^n  \binom{n}{k}\bigg[ \bigg( 2 \sum_{r=1}^{k-1} p_i^r p_j^{k-r} \bigg) +p_i^k +p_j^k \bigg] q_{ij}^{n-k},$$
where $q_{ij} = 1-(p_i +p_j)$.\\

When $p_i = \frac{1}{m}$ for all $i \in [m]$, then $ P(A_{ij}) = 1- \frac{2n(m-1)^{n-1}+(m-2)^n}{m^n}.$
\end{theorem}

\begin{proof}
We will find the probability of the complement, $A_{ij}^c$, which is the event where the two empirical supports do not intersect, i.e., $A_{ij}^c = \{I_n(i) \cap I_n(j)\} = \emptyset$. Let $C_i = \{\ell : Y_\ell = i, 1 \leq \ell \leq n \}$ and define $C_j$ analogously. Note that $A_{ij}^c$ can be expressed as the disjoint union of three events: 
\begin{enumerate}
	\item $\{C_i \text{ and } C_j \text{ are both empty}\}$,
	\item $\{\text{Exactly one of } C_i \text{ or } C_j \text{ is empty}\}$,
	\item $\{C_i \text{ and } C_j \text{ are both non-empty but $I_n(i)$ and $I_n(j)$ do not intersect}\}$.
\end{enumerate}
The probability of the first event is simply $q_{ij}^n$. For the second event, assume for now that $C_i$ will be the non-empty set and let $k \in [n]$ be the desired size of $C_i$. There are $\binom{n}{k}$ ways of choosing the locations of the $k$ points in $C_i$. Once these points are chosen, the probability that these points receive label $i$ and no others receive label $i$ nor label $j$ is exactly $p_i^kq_{ij}^{n-k}$. Summing over all values of $k$ and noting that the argument where $C_j$ is non-empty is analogous, we get that the probability of the second event is exactly $\sum_{k=1}^n \binom{n}{k}(p_i^k +p_j^k)q_{ij}^{n-k}$.

Now, note that the third event only occurs if all the points in $C_i$ are to the left of all points in $C_j$ or vice versa; for now assume $C_i$ is to the left. Let $k\in [n]$ be the desired size of $C_i \cup C_j$ and let $r \in [k-1]$ be the desired size of $C_i$. As before there are $\binom{n}{k}$ ways of choosing the locations of the $k$ points in $C_i \cup C_j$. Once these points are fixed, we know $C_i$ has to be the first $r$ many points, $C_j$ has to be the remaining $k-r$ points, and all other points cannot have label $i$ nor label $j$. This occurs with probability $p_i^r p_j^{k-r}q_{ij}^{n-k}$. Finally, summing over all values of $k$ and $r$ and adding a factor of 2 to account for flipping the sides of $C_i$ and $C_j$ we get that the third event occurs with probability $2\sum_{k=1}^n \binom{n}{k} \sum_{r=1}^{k-1}p_i^r p_j^{k-r}q_{ij}^{n-k}$.

Since $A_{ij}^c$ is the disjoint union of these three events, $P(A_{ij}^c)$ is equal to the sum of these three probabilities, which gives the desired result. For the uniform case, simply set $p_i=p_j=1/m$ in the general formula and note,
\begin{align*}
P(A_{ij}) = & 1- (\frac{m-2}{m})^n -\sum_{k=1}^n  \binom{n}{k}\bigg[ \bigg( 2 \sum_{r=1}^{k-1} \frac{1}{m^k} \bigg) +\frac{2}{m^k}  \bigg] (\frac{m-2}{m})^{n-k}\\
=& 1- (\frac{m-2}{m})^n - \frac{1}{m^n} \sum_{k=1}^n  \binom{n}{k}2k(m-2)^{n-k}\\
= & 1- \frac{2n(m-1)^{n-1}+(m-2)^n}{m^n}.
\end{align*}
\end{proof}

From this we can compute the expected number of edges in the random interval graph, which is the 1-skeleton of $\mathcal{N}_n$. The proof follows immediately from the above by the linearity of expectation. 

\begin{corollary} Let $X$ be the random variable equal to the number of edges in $\mathcal{N}_n$, the random interval graph.  Under the key assumptions in Section \ref{intro},
%Let $Y = (Y_1, \ldots, Y_n)$ be a random vector whose components are i.i.d. random variables such that $P(Y_j = i) = p_i >0$ for all $i \in [m]$. Let $\mathcal{N}_n = \mathcal{N}_n([n], Y)$ denote the empirical nerve of the random coloring induced by $Y$. 
$$ \mathbb{E}X = \hskip -.3cm \sum_{1 \leq i < j \leq m} \hskip -.4cm 1- q_{ij}^n -\sum_{k=1}^n \bigg[ \binom{n}{k} \bigg( 2 \sum_{r=1}^{k-1} p_i^r p_j^{k-r} \bigg) +p_i^k +p_j^k \bigg] q_{ij}^{n-k},$$
where $q_{ij} = 1-(p_i +p_j)$. In the uniform case where $p_i = {{1}\over{m}}$ for all $i\in [m]$, this expectation equals
$$ \binom{m}{2}\bigg( 1- \frac{2n(m-1)^{n-1}+(m-2)^n}{m^n}\bigg).$$
\end{corollary}

%%%%%%%%%%%%%%%%%%%
%%%%%%%%%%%%%%%%%%%
%%%%----Section 4: Cliques, Maximum Degree, and Behavior in the limit
%%%%%%%%%%%%%%%%%%%
%%%%%%%%%%%%%%%%%%%

\section{Maximum Degree, Cliques, and Behavior in the limit.} \label{sec:cliques}

In this section we investigate the connectivity of the empirical nerve. First we study the maximum degree and clique number of $\mathcal{N}_n$. Then we show that as the number of samples $n$ goes to infinity, then $\mathcal{N}_n$ is almost surely the $(m-1)$-simplex.

%%%%%%%%%%%%%%%%%%
%%%%%%% Max Degree
%%%%%%%%%%%%%%%%%%

\subsection{Maximum Degree.}
The following result is a lower bound on the probability of finding an interval intersecting all others, i.e., that the maximum degree $\Deg(\mathcal{N}_n)$ of $\mathcal{N}_n$ is $m-1$. In our birdwatching story this can be interpreted as the probability of finding a species which overlaps in time with all others.

In the following theorem we let $\mathcal{X}_{m,k}^n$ denote the set of weak-compositions of $n$ with length $m$ containing exactly $k$-many non-zero parts \cite[p.~25]{EC1}. Formally, 
$\mathcal{X}_{m,k}^n =\{(x_1,...,x_m)\in \mathbb Z_{\geq 0}^m: \sum_{i=1}^m x_i=n, |\{x_i: x_i \neq 0\}| =k \}$. \\
Also let $M(x)=\frac{(x_1+x_2+...+x_m)!}{x_1!x_2!...x_m!}\prod_{i=1}^m p_i^{x_i}$ denote the multinomial distribution applied to the vector $x\in\mathcal{X}_{m,k}^n$ considering the associated probabilities $p_1,p_2,...,p_m$. Finally, let $S_n^k$ denotes the \emph{Stirling numbers} of the second kind \cite[p.~81]{EC1}.

\begin{theorem} \label{thmmaxdegree} Under the key assumptions in Section \ref{intro}, the maximum degree of  $\mathcal{N}_n$ satisfies
%Let $Y = (Y_1, \ldots, Y_n)$ be a random vector whose components are i.i.d. random variables such that $P(Y_j = i) = p_i >0$ for all $i \in [m]$. Let $\mathcal{N}_n = \mathcal{N}_n([n], Y)$ denote the empirical nerve of the random coloring induced by $Y$. Then\\

$P(\Deg(\mathcal{N}_n)=m-1)\geq$
$$\max_{r}\{[1- \sum_{k=1}^{m-1} \frac{k^r}{m^r}\binom{m}{k}\sum\limits_{x\in\mathcal{X}_{m,k}^r} \hskip -0.3cm M(x)  (m-k)^rp_{*}^r][\hskip -0.2cm \sum\limits_{x\in\mathcal{X}_{m,m}^{n-2r}} \hskip -0.4cm M(x)+\hskip -0.4cm \sum\limits_{x\in\mathcal{X}_{m,m-1}^{n-2r}} \hskip -0.4cm M(x)]\}.$$

Moreover, in the uniform case where $p_i = {{1}\over{m}}$ for all $i\in [m]$, we have that\\
%\smallskip

%$P(\Deg(\mathcal{N}_n)=m-1)\geq$\\ $\max_{r}\{[1-\frac{1}{m^{2r}} \sum\limits_{k=1}^{m-1} \binom{m}{k}k!S(r,k)(m-k)^r][\frac{m!}{m^{n-2r}}$ $S(n-2r,m)+\frac{(m-1)!}{(m-1)^{n-2r}}S(n-2r,m-1)]\}.$\\

$P(\Deg(\mathcal{N}_n)=m-1)\geq$ $$\max_{r}\{[1-\frac{m!}{m^{2r}} \sum\limits_{k=1}^{m-1} \frac{(m-k)^r}{(m-k)!}S_r^k][\frac{m!}{m^{n-2r}}S_{n-2r}^{m}+\frac{(m-1)!}{(m-1)^{n-2r}}S_{n-2r}^{m-1}]\}.$$

\end{theorem}

\begin{proof}
Fix some $1\leq r \leq \frac{n-m}{2}$ and consider the sets $L=\{1,2,...,r\}$, $C=\{r,r+1,...,n-(r+1)\}$ and $R=\{n-r,n-(r-1),...,n\}$. If the following events hold, we can guarantee that $\Deg(\mathcal{N}_n)=m-1$.\\
$A=$ $\{$There exists a chromatic class with points in $L$ and $R\}$,\\
$B=$ $\{$There exists at least one point with each of the remaining $m-1$ colors in $C\}$. \\

In order to calculate $P(A)$, we will compute the probability of its complement $A^c$, i.e., the event where no color appears in both $L$ and $R$.
First we calculate the probability of $L$ being colored with exactly $k$ colors with $1\leq k \leq m-1$. Observe that there are $\binom{m}{k}$ ways to choose  these colors and $k^r \sum_{x\in\mathcal{X}_{m,k}^r} M(x)$ ways to color $L$ with them. As there exist $m^r$ different colorings with all the $m$ colors, we have that for a fixed $k$ the probability is $$\frac{1}{m^r}k^r\binom{m}{k}\sum\limits_{x\in\mathcal{X}_{m,k}^r}\hskip -.25cm M(x).$$ 

%\textcolor{red}
In order for $A^c$ to occur, we need that $R$ be colored with only the $(m-k)$ remaining colors. Note that this event is independent from the coloring of $L$ as the two sets are disjoint. There are $(m-k)^r$ different ways of coloring $R$, and each occurs with probability at most $p_{*}^r$, where $p_*=\max\{p_i:1\leq i \leq m\}$. Thus, for  a fixed $ k $ we have that the probability that no color appears in both $L$ and $R$ is at most
$$\left[ \frac{1}{m^r}k^r\binom{m}{k}\sum\limits_{x\in\mathcal{X}_{m,k}^r} \hskip -.25cm M(x)\right] \left[ (m-k)^rp_{*}^r\right].$$

Then, by summing over all $k$ we have that

$$P(A^c)\leq \sum_{k=1}^{m-1}\left[ \frac{1}{m^r}k^r\binom{m}{k}\sum\limits_{x\in\mathcal{X}_{m,k}^r} \hskip -.25cm M(x)\right] \left[ (m-k)^rp_{*}^r\right] ,$$

which implies that

$$P(A)\geq 1- \sum_{k=1}^{m-1} \frac{1}{m^r}k^r\binom{m}{k}\sum\limits_{x\in\mathcal{X}_{m,k}^r} \hskip -.25cm M(x)  (m-k)^rp_{*}^r.$$

To compute $P(B)$, note that the probability of coloring $C$ with $m$ or $m-1$ colors is exactly
$$\sum\limits_{x\in\mathcal{X}_{m,m}^{n-2r}} \hskip -.4 cm M(x)+\sum\limits_{x\in\mathcal{X}_{m,m-1}^{n-2r}} \hskip -.5cm M(x).$$

Finally, as $A$ and $B$ are independent events, we have $P(\Deg(\mathcal{N}_n)=m-1)$ is greater than
$$[1- \sum_{k=1}^{m-1} \frac{1}{m^r}k^r\binom{m}{k}\sum\limits_{x\in\mathcal{X}_{m,k}^r} \hskip -.3cm M(x)  (m-k)^rp_{*}^r][\sum\limits_{x\in\mathcal{X}_{m,m}^{n-2r}} \hskip -.4 cm M(x)+\sum\limits_{x\in\mathcal{X}_{m,m-1}^{n-2r}} \hskip -.5cm M(x)].$$

Maximizing over $r$ gives the desired result. For the case uniform, we just apply $p_{*}=1/m$ and use the former equality together with the fact that $k!/k^nS_n^k=\sum_{x\in\mathcal{X}_{m,k}^n}M(x)$.
\end{proof}

%%%%%%%%%%%%%%%%%%%
%%%%---- Cliques 
%%%%%%%%%%%%%%%%%%%
\subsection{Cliques.}
The expected clique number of $\mathcal{N}_n$ is of special interest to us. In our birdwatching story this corresponds to the maximal subset of birds whose time intervals all intersect. While we do not compute the expected clique number exactly, the next theorem provides a lower bound on the expected clique number which performs very well in simulations (see the Appendix).

\begin{lemma}  Under the key assumptions in Section \ref{intro},
%Let $Y = (Y_1, \ldots, Y_n)$ be a random vector whose components are i.i.d. random variables such that $P(Y_j = i) = p_i >0$ for all $i \in [m]$. 
 the probability that an arbitrary point $x\in [n]$ lies inside the interval of color $i$, $I_n(i)$, is exactly 
% $$ 1-[(1-p_i)^{x}+ (1-p_i)^{n-x+1} -(1-p_i)^{n}].$$
$ 1-q_i^{x}- q_i^{n-x+1} +q_i^{n},$
where $q_i=1-p_i.$
\end{lemma}

\begin{proof}
Fix an arbitrary $x\in [n]$ and define the event $A = \{x\in I_n(i)\}$. Note that in order for $A$ to occur either $x$ lies between two points with label $i$ or $x$ itself is labeled $i$. Now consider the complementary probability event, $A^c = \{x \notin I_n(i)\}$. Next define the events $L,R$ where $L=\{$none of the points \emph{less than or equal} to $x$ have label $i\}$ and $R=\{$none of the points \emph{greater than or equal to} $x$ have label $i\}$. Note $A^c = L \cup R$ and $P(L)=q_i^{x}$, $P(R)=q_i^{n-x+1}$ and $P(L\cap R)=p_i^{n}$. Therefore, by the inclusion-exclusion principle we have, 

$$P(A^c)= P(L)+P(R)-P(L \cap R)   = q_i^{x}+ q_i^{n-x+1} -q_i^{n},
$$ 
and hence
$P(A)= 1-q_i^{x}- q_i^{n-x+1} +q_i^{n}.$
\end{proof}

\begin{theorem} \label{expectedcliquenum}  
%Let $Y = (Y_1, \ldots, Y_n)$ be a random vector whose components are i.i.d. random variables such that $P(Y_j = i) = p_i >0$ for all $i \in [m]$. Let $\mathcal{N}_n = \mathcal{N}_n([n], Y)$ denote the empirical nerve of the random coloring induced by $Y$, and
 Let $\omega$ be the random variable equal to the clique number of $\mathcal{N}_n$, i.e., the size of the largest clique in the 1-skeleton of $\mathcal{N}_n$. Under the key assumptions in Section \ref{intro} we have
%\begin{center}
%$\mathbb{E }$ $\omega \geq \sum\limits_{i=1}^m \bigg(  1-[(1-p_i)^{\lceil \frac{n}{2}\rceil}+ (1-p_i)^{n-\lceil \frac{n}{2}\rceil+1} -(1-p_i)^{n}]\bigg).$
%\end{center}
\begin{center}
$\mathbb{E }$ $\omega \geq \sum\limits_{i=1}^{m}(  1-q_i^{\lceil \frac{n}{2}\rceil}-q_i^{n-\lceil \frac{n}{2}\rceil +1}+q_i^{n})$
\end{center}
where $q_i=1-p_i$. Moreover, in the uniform case where $p_i= {{1}\over{m}}$ for all $i\in [m]$, we have that
\begin{center}
$\mathbb{E }$ $\omega \geq m - \big(\frac{m-1}{m}\big)^{\lceil \frac{n}{2}\rceil} - \big(\frac{m-1}{m}\big)^{n-\lceil \frac{n}{2}\rceil +1} + \big(\frac{m-1}{m}\big)^{n}.$
\end{center}\end{theorem}

\begin{proof}
By the preceding lemma we know that the probability that $x\in I_n(i)$ for some $x\in [n]$ is exactly
$1-q_i^{x}- q_i^{n-x+1} +q_i^{n}.$
To maximize this quantity over $x \in [n]$ we will first minimize $f(x) = q_i^{x}+ q_i^{n-x+1} -q_i^{n}$ over all $x$. Note $f$ is convex so a simple calculus exercise shows that $f$ is minimized at  $x^* = \frac{n+1}{2}$. This can also be seen directly from the fact that $f$ is convex and symmetric about $\frac{n+1}{2}$.
%\begin{align*}
%\frac{\partial f}{\partial x} &= (1-p_i)^{x^*}\log(1-p_i) -(1-p_i)^{n-x^*+1}\log(1-p_i) =0  \\
%&\iff \log(1-p_i)[(1-p)^{x^*} - (1-p_i)^{n-x^*+1}] =0 \\
%&\iff (1-p)^{x^*} = (1-p_i)^{n-x^*+1} \\
%&\iff x^* = \frac{n+1}{2}.
%\end{align*}
When $n$ is odd the minimizer $x^*$ is an integer and lies in $[n]$. To handle the case when $n$ is even, note that $f$ is symmetric about the minimizer $x^*$. Therefore,
%\begin{align*}
%f(x^* -h) &= (1-p_i)^{\frac{n+1}{2}-h}+ (1-p_i)^{n - (\frac{n+1}{2} -h) +1} -(1-p_i)^{n} \\
%&= (1-p_i)^{\frac{n+1}{2}-h}+ (1-p_i)^{\frac{2n -n-1+2}{2} +h} -(1-p_i)^{n} \\
%&= (1-p_i)^{\frac{n+1}{2}-h}+ (1-p_i)^{\frac{n+1}{2}+h} -(1-p_i)^{n}.
%\end{align*}
when $n$ is even, $f$ is minimized over $[n]$ at the integers closest to $x^*$, which are $\frac{n}{2}$ and $\frac{n}{2}+1$. We see then that $f$ is minimized over $[n]$ at the point $x = \lceil \frac{n}{2} \rceil$, which holds whether $n$ is even or odd.

Now, for $i = 1, \ldots, m$ let $X_i$ be the indicator random variable which equals $1$ if $ \lceil \frac{n}{2} \rceil \in I_n(i)$ and is $0$ otherwise and set $X = \sum_{i=1}^m X_i$, so $X$ counts the number of intervals containing the point $ \lceil \frac{n}{2} \rceil$.  Note that the clique number $\omega \geq X$, so
%\begin{align*}
%\mathbb{E } \text{ } \omega \geq \mathbb{E} X =  \sum\limits_{i=1}^{m} \mathbb{E} X_i
%&=\sum\limits_{i=1}^{m}P(X_i) \\
%&=\sum\limits_{i=1}^{m}\bigg(  1-[(1-p_i)^{\lceil \frac{n}{2}\rceil}+ (1-p_i)^{n-\lceil \frac{n}{2}\rceil +1} -(1-p_i)^{n}]\bigg).
%\end{align*}
$$\mathbb{E } \text{ } \omega \geq \mathbb{E} X =  \sum\limits_{i=1}^{m} \mathbb{E} X_i
=\sum\limits_{i=1}^{m}P(X_i)=\sum\limits_{i=1}^{m}(  1-q_i^{\lceil \frac{n}{2}\rceil}-q_i^{n-\lceil \frac{n}{2}\rceil +1}+q_i^{n}).$$

The result for the uniform case follows directly by setting $p_i =\frac{1}{m}$ for every $i$.
\end{proof}

%%%%%%%%%%%%%%%%%%
%%%%%%% Behavior as n --> infinity
%%%%%%%%%%%%%%%%%%
\subsection{Behavior of the nerve complex as the number of samples goes to infinity.}
Note that as the number of samples $n$ grows large, Theorem \ref{expectedcliquenum} implies that the expected clique number $\mathbb{E }$ $\omega \to m$. Since $\omega$ only takes values in $\{1, \ldots, m\}$ it follows that the clique number also converges to $m$ in probability. Thus,
 as $n$ goes to infinity, the probability that the nerve of the observations is the $(m-1)$-simplex denoted by $\Delta_{m-1}$, i.e., a complete graph, goes to 1. In our birdwatcher analogy, this implies that with sufficiently many observations one is almost sure to find a common interval of time where all $m$ species can be observed. 
The following theorem provides a lower bound on this convergence. 

\begin{theorem}\label{probtosimplex} Under the key assumptions in section \ref{intro}, the probability that $\mathcal{N}_n$ is an $(m-1)$-simplex (or equivalently the graph is a complete graph $K_m$) satisfies
%Let $Y = (Y_1, \ldots, Y_n)$ be a random vector whose components are i.i.d. random variables such that $P(Y_j = i) = p_i >0$ for all $i \in [m]$. Let $\mathcal{N}_n = \mathcal{N}_n([n], Y)$ denote the empirical nerve of the random coloring induced by $Y$, then 
$$P(\mathcal{N}_n = \Delta_{m-1}) \geq ( \sum\limits_{x \in \mathcal{X}_{m}^{\lfloor\frac{n}{2}\rfloor}} \hskip -.25cm M(x))^2$$
where $\mathcal{X}_m^{\lfloor\frac{n}{2}\rfloor} = \{(x_1,x_2,...,x_m)\in \mathbb{N}^m:\sum_{i=1}^m x_i= \lfloor \frac{n}{2} \rfloor \}.$\\ 
In the uniform case where $p_i = \frac{1}{m}$ for every $i\in[m]$, this gives that 
$$ 
P(\mathcal{N}_n = \Delta_{m-1}) \geq  \left( \frac{m!}{m^{\lfloor \frac{n}{2} \rfloor }} S_{\lfloor \frac{n}{2} \rfloor}^{m} \right)^2
$$
where, again, $S_n^k$ denotes the Stirling numbers of the second kind.
\end{theorem}

\begin{proof} 

%{\color{red} QUE COSA ES $\mathcal{X}$??????? POR FAVOR CORREGIR ESTO!!}
For each vector $x\in\mathcal{X}_m^{\lfloor\frac{n}{2}\rfloor}$ the multinomial $M(x)$ computes the probability that there exist exactly 
$x_i$ points with color $i$ for every $1\leq i \leq m$. Therefore, the sum over all the vectors of $\mathcal{X}_m^{\lfloor\frac{n}{2}\rfloor}$ gives us the probability of having at least one point of each color.\\
Now, we consider the events $L=$ $\{$the first $\lfloor \frac{n}{2} \rfloor$ points are colored with exactly $m$ colors$\}$ and
$R=$ $\{$the last $\lfloor \frac{n}{2} \rfloor $ points are colored with exactly $m$ colors$\}$. We have $$P(L)=P(R)=\sum\limits_{x \in \mathcal{X}_m^{\lfloor \frac{n}{2}\rfloor}} \hskip -.25cm M(x).$$
Then $P(\mathcal{N}_n=\Delta_{m-1})\geq P(L\cap R)$ and as $L$ and $R$ are independent events, we conclude
$$P(\mathcal{N}_n=\Delta_{m-1})\geq P(L\cap R)=P(L)P(R)=( \sum\limits_{x \in \mathcal{X}_m^{\lfloor \frac{n}{2}\rfloor}} \hskip -.25cm M(x))^2.$$

The result for the uniform case follows because $k!/k^nS_n^k=\sum_{x\in\mathcal{X}_{m,k}^n}M(x)$.
\end{proof}

Theorem \ref{probtosimplex} tells us how likely it is for the empirical nerve of $n$ samples to form the $(m-1)$-simplex for fixed $n$. A related question asks what is the \emph{first} observation $n$ for which this occurs, i.e., if we have a sequence of observations $Y_1, Y_2, \ldots$ what is the least $n$ such that $\mathcal N_n((Y_1, \ldots, Y_n)) = \Delta_{m-1}$? We call this quantity the \emph{waiting time} to form the $(m-1)$-simplex and provide a lower bound on its expectation below. 

\begin{theorem}\label{waitingtime} 
%Let $Y= Y_1, Y_2, \ldots$ be a sequence i.i.d. random variables such that $P(Y_j = i) = p_i >0$ for all $i \in [m]$. Let $\mathcal{N}_n = \mathcal{N}_n([n], Y)$ denote the empirical nerve of the random coloring induced by the first $n$ variables, $(Y_1, \ldots, Y_n)$. 
Let $X$ be the random variable for the waiting time until $\mathcal N_n = \Delta_{m-1}$, explicitly $X = \inf\{n \in \mathbb N : \mathcal N_n = \Delta_{m-1} \}$. Then, under the key assumptions in Section \ref{intro}, we have
$
\mathbb E X \leq 2 \int_0^\infty \Big(1- \prod_{i=1}^m (1-e^{-p_ix}) \Big) dx.
$
Moreover, in the uniform case, where $p_i = {{1}\over{m}}$ for all $i\in [m]$, we have that
$
\mathbb E X \leq 2m \sum_{i=1}^m \frac{1}{i}.
$
\end{theorem}
\begin{proof}
The results follow directly from the expected waiting time of the classical coupon collector problem. Let $Z$ denote the waiting time until we have observed every label, i.e., $Z$ is the waiting time until we have completed a collection of coupons if each coupon is an i.i.d. random variable that takes value $i$ with probability $p_i$. It is known that $\mathbb E Z = 2 \int_0^\infty \big(1- \prod_{i=1}^m (1-e^{-p_ix}) \big) dx$, and in the uniform case where $p_i =  {{1}\over{m}}$ for all $i\in [m]$, $\mathbb E Z = m \sum_{i=1}^m \frac{1}{i}$  (see \cite{Coupon} for several detailed proofs). Now, note that $\mathcal N_n = \Delta_{m-1}$ if we complete a collection, then complete a second collection, disjoint from the first.  Let $Z_1$ denote the waiting time to complete the first collection, and let $Z_2$ be the additional waiting time to complete a second collection. Then $X \leq Z_1 + Z_2$ and $Z_1, Z_2$ are equal in distribution to $Z$, so $\mathbb E X \leq \mathbb E(Z_1 +Z_2) = 2\mathbb E Z$. 
\end{proof}

%%%%%%%%%%%%%%%%%%%%%%%%% NEW ENDING %%%%%%%%%%%%%%%%%%

\section{Conclusion.} \label{conclusiones}

In this paper we introduced  a novel random interval graph model. It is well-suited for applications involving the overlap patterns of chronological observations. There are a number of natural fascinating questions for the curious reader.
First, obviously the distribution of birds  varies in time as seasonal changes (or other factors such as predators or climate change) affect the species, thus the non-stationary case is better for applications. We ask ourselves, which of the results can be extended to the non-stationary case when the key assumptions made here are no longer valid?
%{\color{red}What can be extended to the non-stationary case when the key assumptions are not valid anymore?}

Second, Hanlon presented in \cite{hanlon82} a characterization of all interval graphs using a unique interval representation. He used this to enumerate all interval graphs. The analysis we presented in Theorem \ref{probtosimplex} indicates that, when we use our stochastic process to generate random intervals on the line, the probability of getting an interval graph other than the complete graph goes to 0 as the number of samples $n$ goes to infinity. A natural challenge is to understand the decay of probabilities for different classes of graphs, for instance, random \emph{interval trees} (see \cite{EckhoffIntervalGraph}). 
%These questions show that there is more to watch for.
 
Finally, the story we presented is about data samples indexed by a single parameter, say time. But what happens when geographical coordinates, temperature, humidity, or other parameters are considered to model the distribution of birds? Extending the model to higher-dimensions produces new challenges. For example, the random interval graphs are no longer sufficient to capture all the information. Instead, one needs to investigate random simplicial complexes (see \cite{Hogan_Tverberg2020}) because we lose the natural order for the points that we have in the line. This implies that an equivalent result to Lemma 1 is no longer possible. For instance, continuing with our birdwatcher's analogy, suppose that colored points in Figure \ref{dib} represent the geographic coordinates of three different types of birds that have been studied. If our birdwatcher is trying to determine the usual habitat and the territorial interactions between them he/she will face the problem that two very similar data sets will induce different simplicial complexes.  

\begin{figure}[h!]
\begin{center}
\includegraphics[scale=.18]{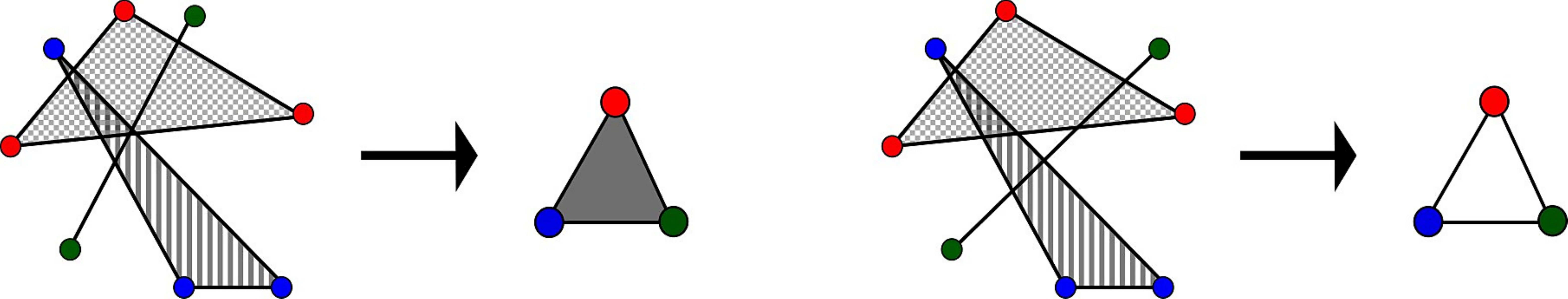}
\caption{Two data sets of $3$ different bird species with the same order type inducing different simplicial complexes.}
\label{dib}
\end{center}
\end{figure}

\noindent {\bf Acknowledgements.} The first and second authors     %Edgar and Jesus  
gratefully acknowledge partial support from NSF DMS-grant 1818969. The second author also acknowledges support from the NSF-AGEP supplement. Finally, the third and fourth authors %Deborah and Antonio 
gratefully acknowledge partial support from PAPIIT IG100721 and CONACyT 282280.  
\vskip 0.8cm
%\bibliographystyle{maa1.bst}
%\bibliography{references__1}

\def\cprime{$'$}

%%%%%%%%%%%%%%%%%%%%%%%%%%%%%%%%
%%%%%%%%%%%%%%%%%%%%%%%%%%%%%%%%

\textbf{J.~A. De Loera} is a professor of Mathematics at the University of California, Davis. His main mathematical themes are discrete geometry and combinatorial optimization. He enjoys walking with his dog Bolo and watching coyotes roam the fields near his house. 

Department of Mathematics, University of California, Davis\\
deloera@math.ucdavis.edu

\textbf{E. Jaramillo-Rodriguez} is a PhD candidate in Mathematics at the University of California, Davis. Edgar is writing their thesis on stochastic combinatorial geometry applied to data science and machine learning. Edgar likes spending time outdoors with good friends or good books.

Department of Mathematics, University of California, Davis\\
ejaramillo@ucdavis.edu

\textbf{D. Oliveros} is a professor at the Institute of Mathematics at the National Autonomous University of Mexico UNAM (Campus Juriquilla). 
Her areas of interest in mathematics are discrete and computational geometry and convexity. 
She enjoys dancing, gardening and playing with her dogs.

Instituto de Matem\'aticas, Universidad Nacional Aut\'onoma de M\'exico\\
doliveros@im.unam.mx

\textbf{A.~J. Torres} is a doctoral student in Mathematics at the National Autonomous University of Mexico UNAM. His main areas of interest include discrete geometry, data analysis, and combinatorics. He enjoys jogging around the city and hiking the trails near his hometown in Quer\'etaro.

Instituto de Matem\'aticas, Universidad Nacional Aut\'onoma de M\'exico\\
antonio.torres@im.unam.mx

\vfill\eject

%%%%%%%%%%%%%%%%%%%%%%%%%%%%%
%%%%%%%%%%%%%%%%%%%%%%%%%%%%%%%%%%

%\newpage
\section{Appendix.}

% EDITADO: APPENDIX QUE QUITAMOS
%\subsection{On uncountable collections of independent variables} In Section \ref{intro} we referenced the fact that one of our assumptions for the special case comes with some peril. Namely, we assumed that for any pair $s\neq t$ in our indexing set $I$, the variables $Y_s$ and $Y_t$ are independent. The issue with this statement is that such a collection of random variables cannot be constructed on the standard probability space: $\Omega = [0,1]$, $\mathcal F = \mathcal B([0,1])$, $P = \mu_{leb}$ (except in the trivial case where they are all constant). One proof of this fact is as follows: suppose such a collection did exist and denote it $\{ X_i : i\in I\}$. Note that each $X_i$ is bounded, as we assumed they take values in $[m]$, and so $X_i \in L^2(\Omega, P)$. Thus the vectors $Z_i :=  (X_i - \mathbb EX_i)$ form an uncountable set of independent, and hence orthogonal variables in $ L^2(\Omega, P)$. However this contradicts the fact that $ L^2(\Omega, P)$ is separable, because any separable Hilbert space admits at most a countable orthonormal basis. 

%Luckily, this challenge does not affect our analysis for the following reasons. Firstly, one can show that such a family of variables can be constructed on a non-standard probability space (this is a consequence of the Kolmogorov extension theorem, see \cite{durrett_2019}). Secondly, in our analysis we only ever use \emph{finitely many} of these random variables so the construction is in fact unnecessary for understanding this manuscript. 

\subsection{Experimental Results.}
In Theorems \ref{thmmaxdegree}, \ref{expectedcliquenum}, and \ref{probtosimplex} we provided lower bounds on the likelihood of various events occurring given $n$ points and $m$ labels. To study the usefulness of these bounds we ran simulations. For each pair $(m,n)$ we randomly colored $n$ points on the real line using $m$ colors with uniform probability (each color was equally likely) then constructed the induced interval graph. We repeated this process 100 times for each pair $(m,n)$ and plotted the percentage of the simulations where the desired event occurred. We also plotted our lower bounds from the theorems above and found that, in general, our bounds perform well for most values of $m$ and $n$. 

Figure \ref{fig:Delta m-1} compares the bound on the maximum degree obtained in Theorem \ref{thmmaxdegree} and the empirical approximation generated by our simulations. Figure \ref{fig:clique_number} compares the bound on the expected clique number obtained in Theorem \ref{expectedcliquenum} and the empirical approximation generated by our simulations. Figure \ref{fig:complete_bounds} compares the bound on the probability of the nerve being the $(m-1)$ simplex obtained in Theorem \ref{probtosimplex} and the empirical approximation generated by our simulations. 

Finally, we also compared the probability of obtaining $m-1$ as a maximum degree $\Deg(\mathcal{N}_n)$ in the Scheinerman model with our model. In \cite{JusticzScheinermanWinkler}, the authors prove in a clever way, that this probability is exactly $2/3$ and their result does not depend on the number of intervals. On the other hand, in our model this probability depends on both, the number of points and the number of colors that we use.
We generated $1000$ random $m-$colorings, for Scheinerman's model (the solid line) $1\leq m \leq 50$. For our model we use several values of $n$ with $1\leq m \leq n$. The results are displayed in Figure \ref{fig:degree}.

\begin{figure}[hbt]
%\subfigure{\includegraphics[width=2.7cm, height=3.5cm]{5col_m_optimo.png}} \hskip -.7cm
\subfigure{\includegraphics[width=3cm, height=3cm]{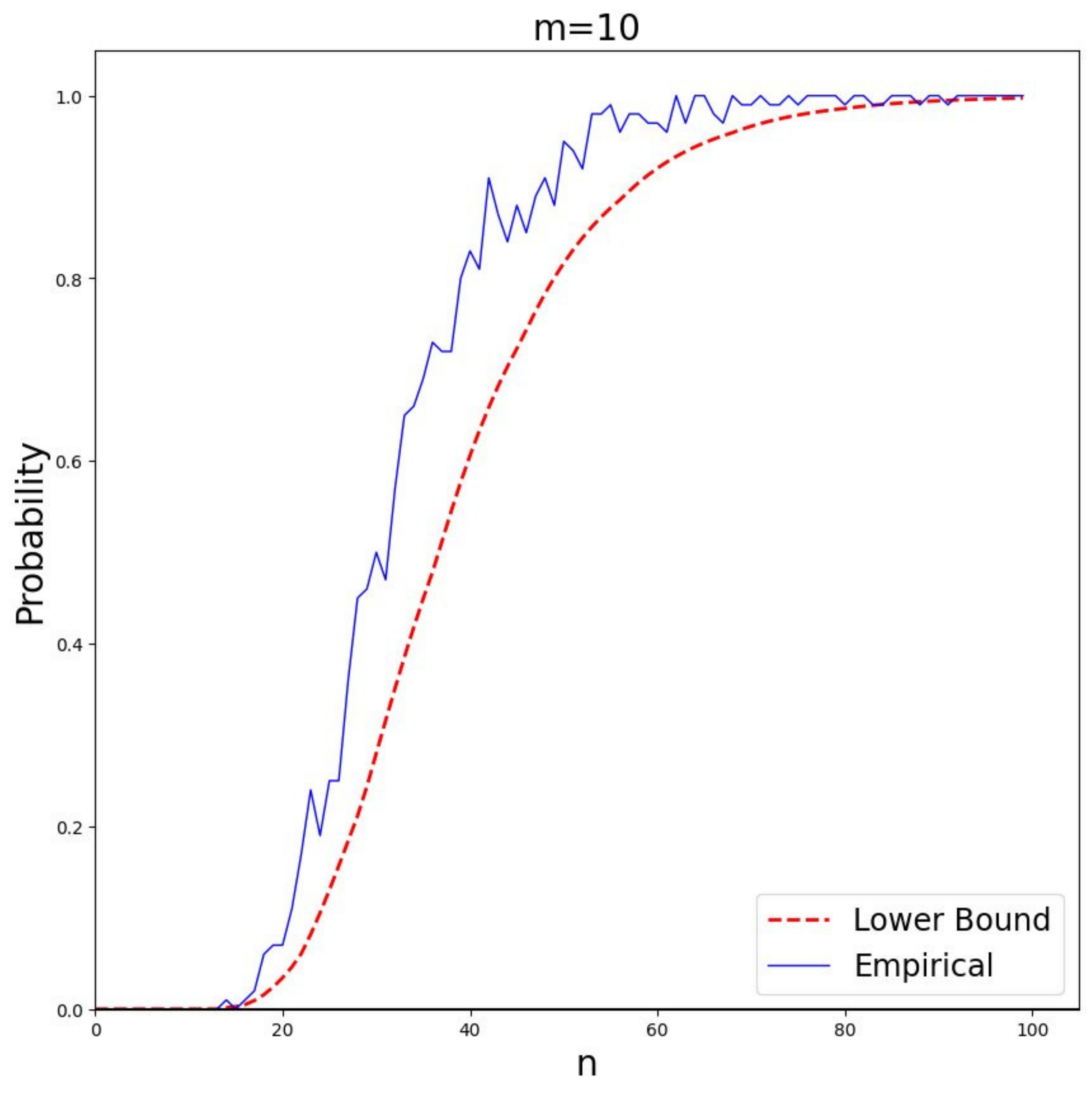}}
\subfigure{\includegraphics[width=3cm, height=3cm]{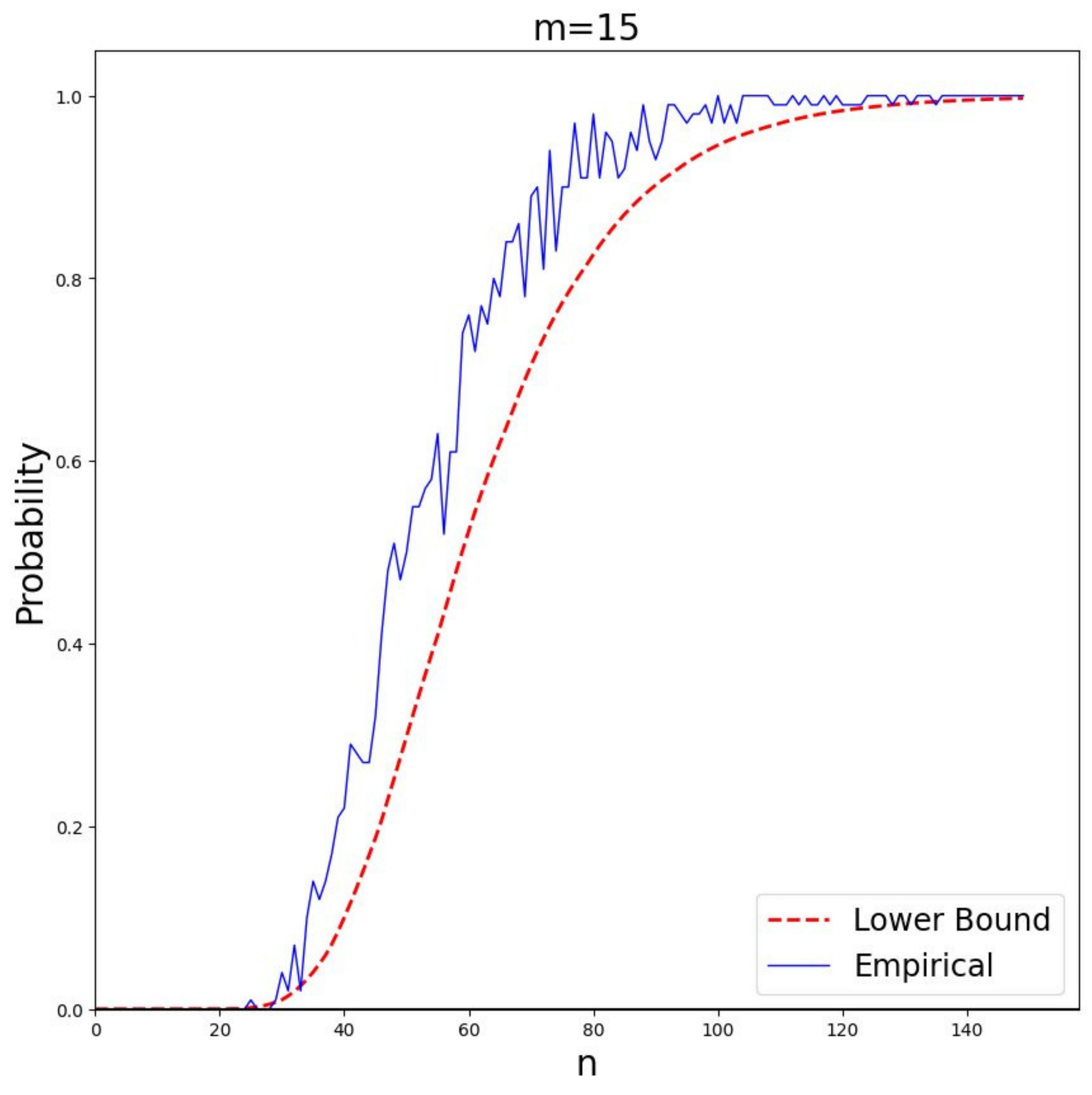}}
\subfigure{\includegraphics[width=3cm, height=3cm]{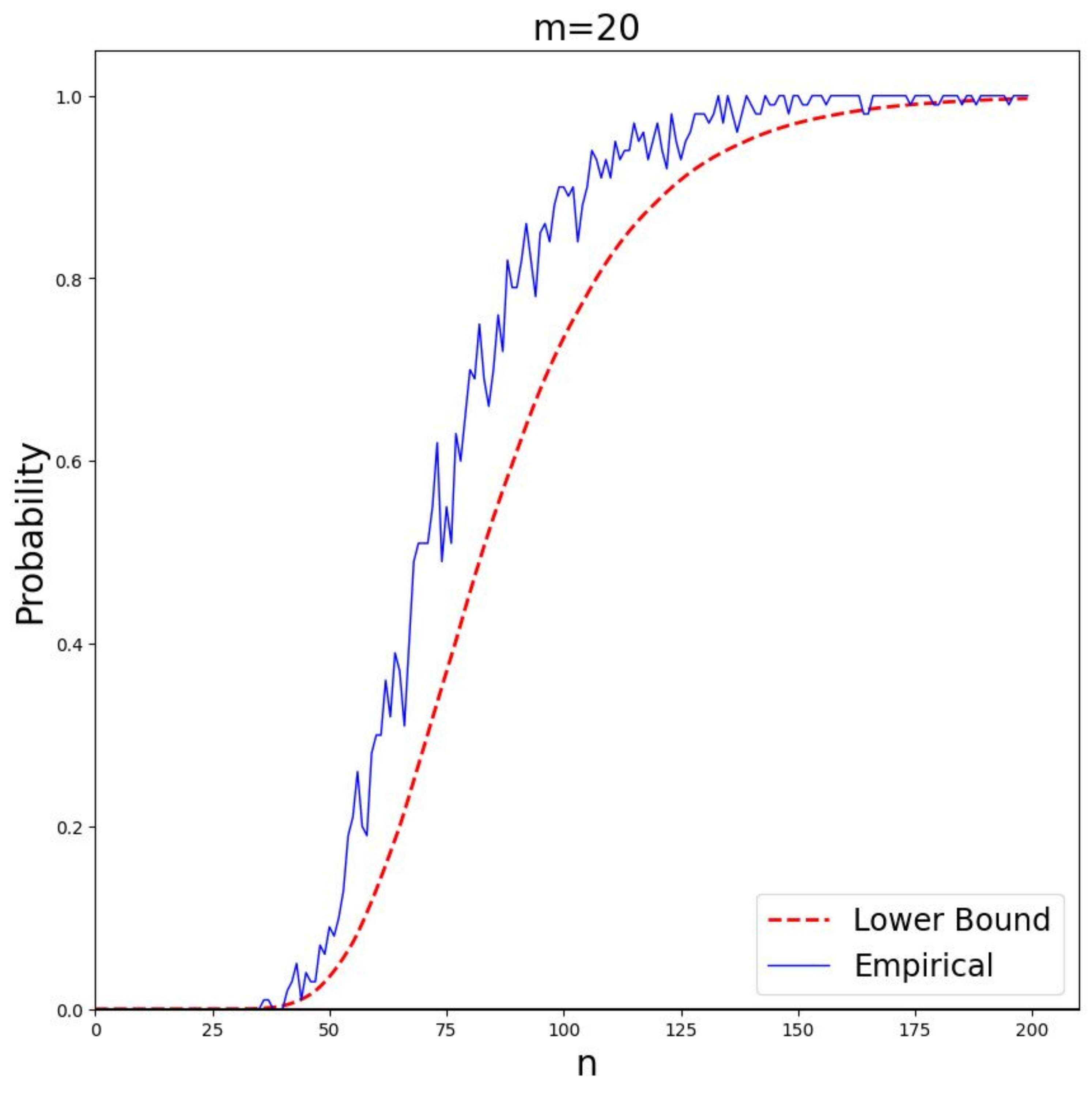}}
\subfigure{\includegraphics[width=3cm, height=3cm]{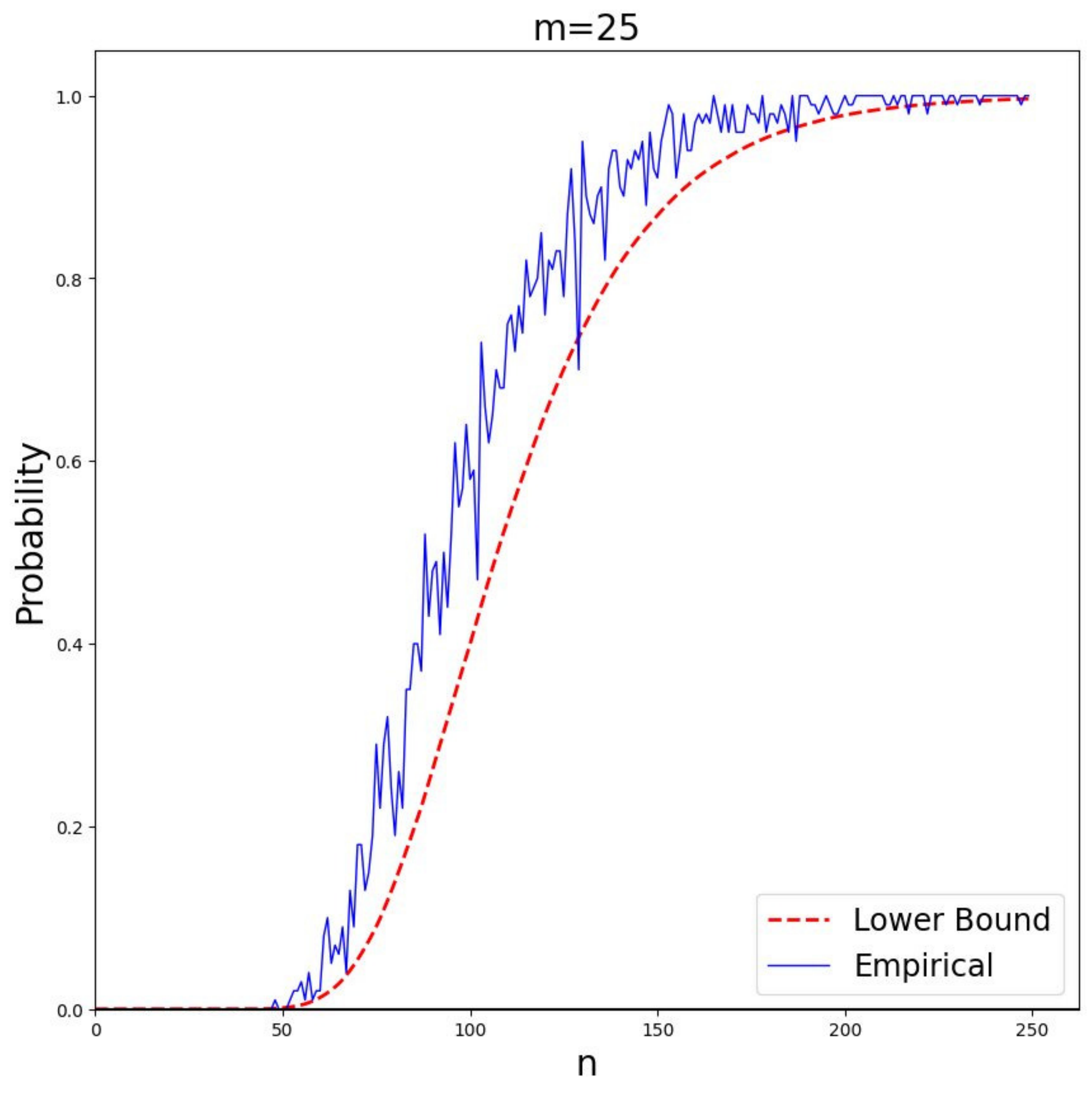}} \\
\subfigure{\includegraphics[width=3cm, height=3cm]{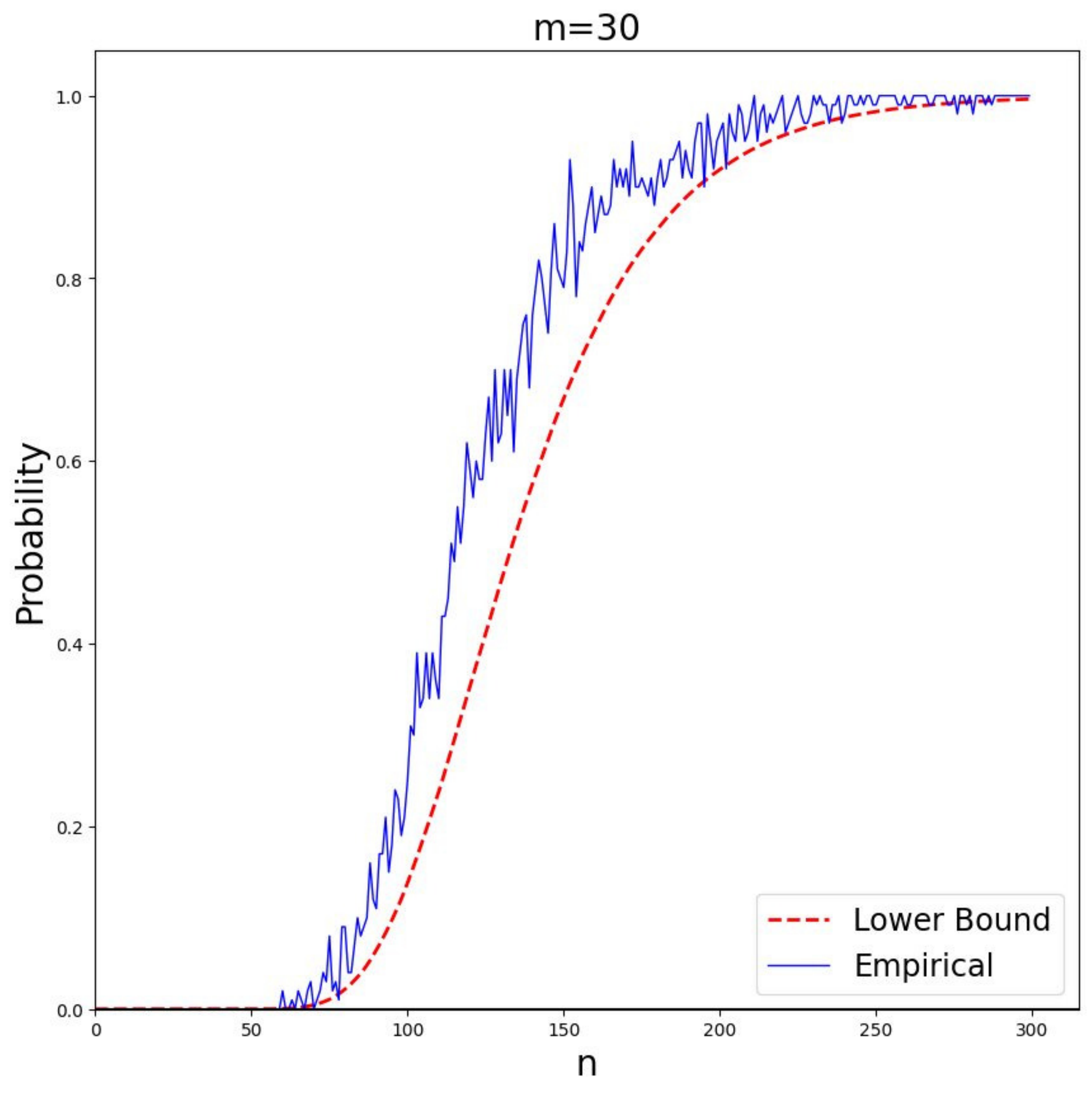}} 
\subfigure{\includegraphics[width=3cm, height=3cm]{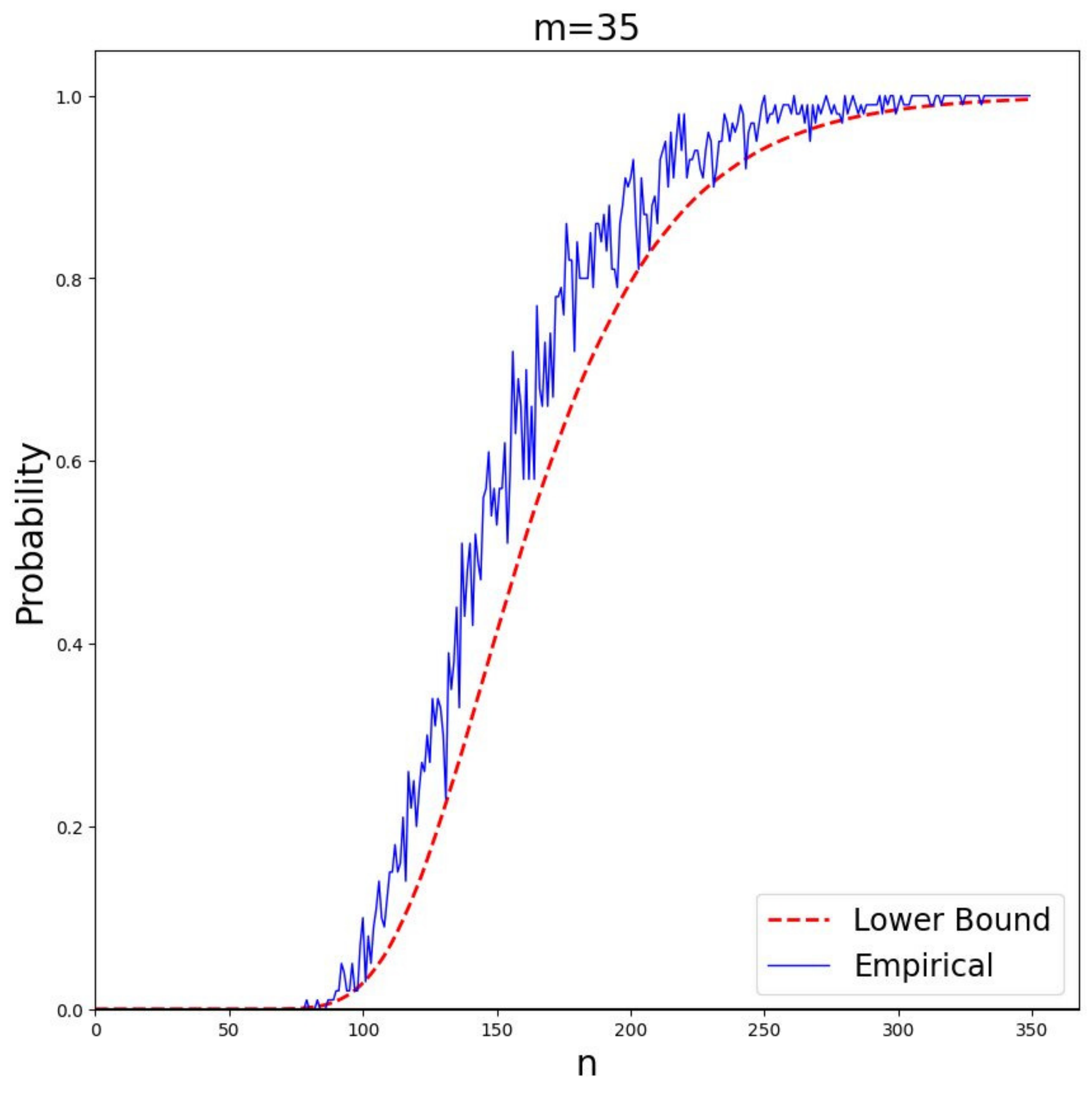}} 
\subfigure{\includegraphics[width=3cm, height=3cm]{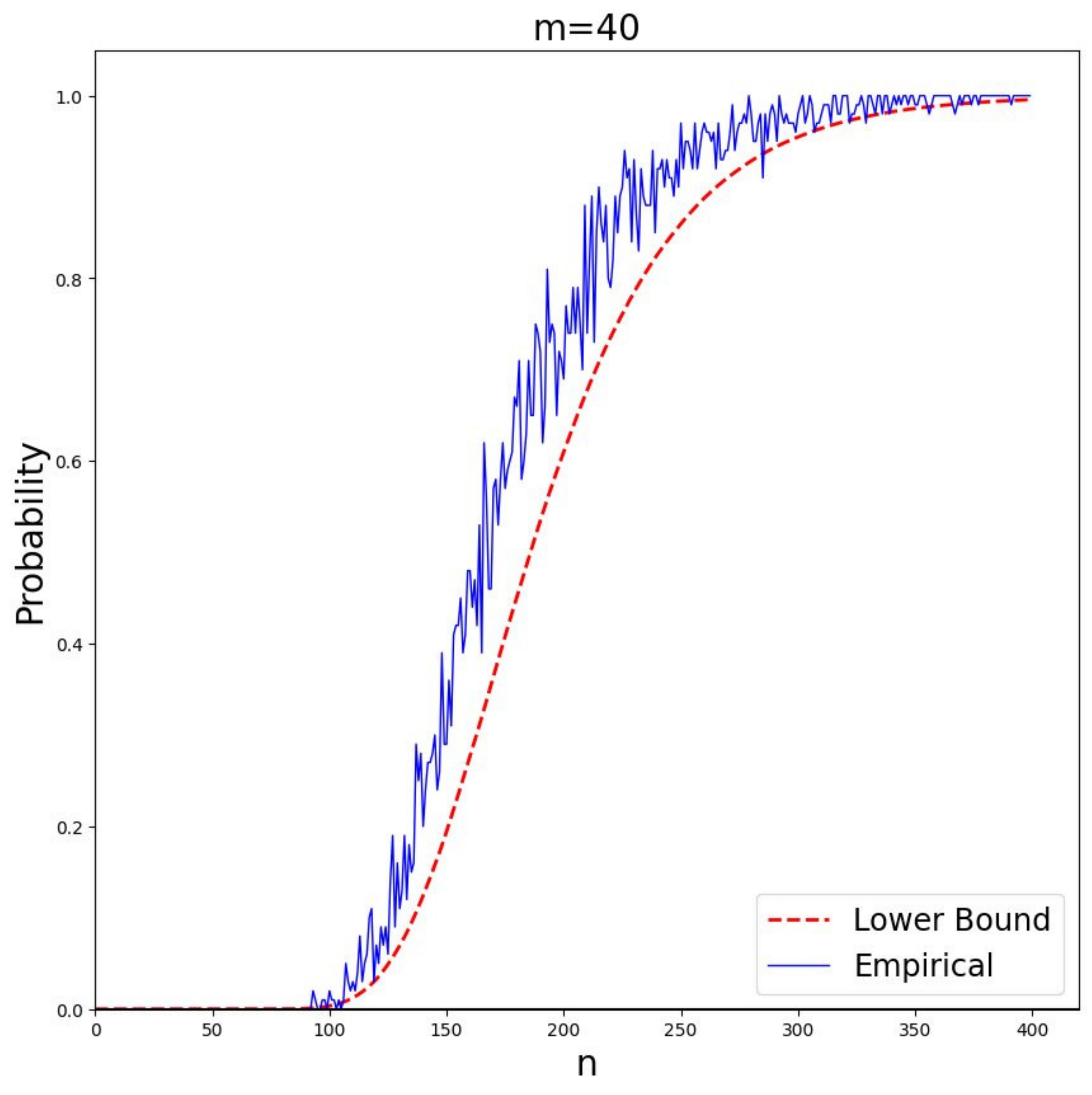}} 
\subfigure{\includegraphics[width=3cm, height=3cm]{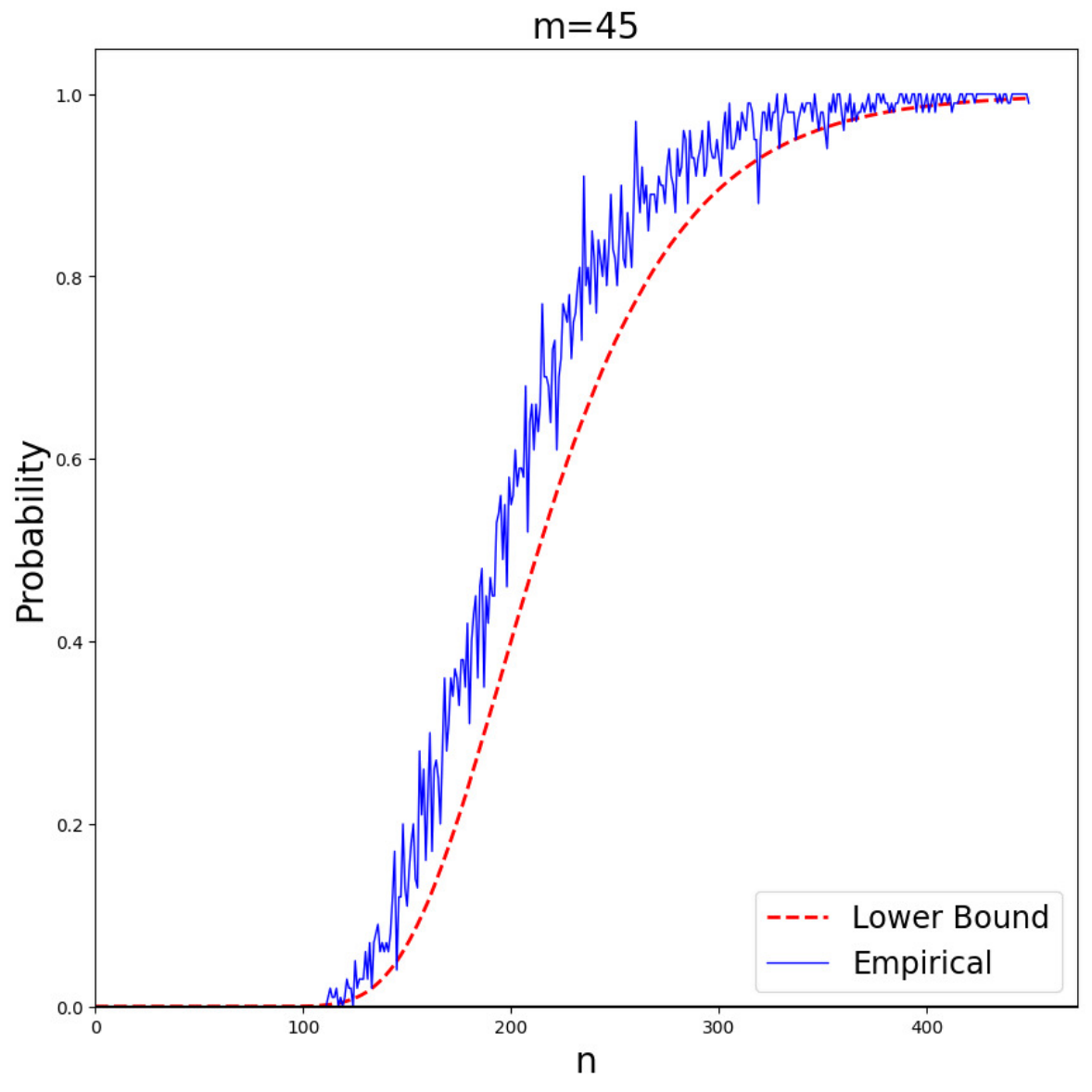}} 
%\subfigure{\includegraphics[width=2.7cm, height=3.5cm]{50col_m_optimo.png}}
\caption{Probability of $\Deg(\mathcal{N}_n) = m-1$, simulations compared to bound from Theorem \ref{thmmaxdegree}.}
\label{fig:Delta m-1}
\end{figure}

\begin{figure}[hbt]
\includegraphics[width=\textwidth]{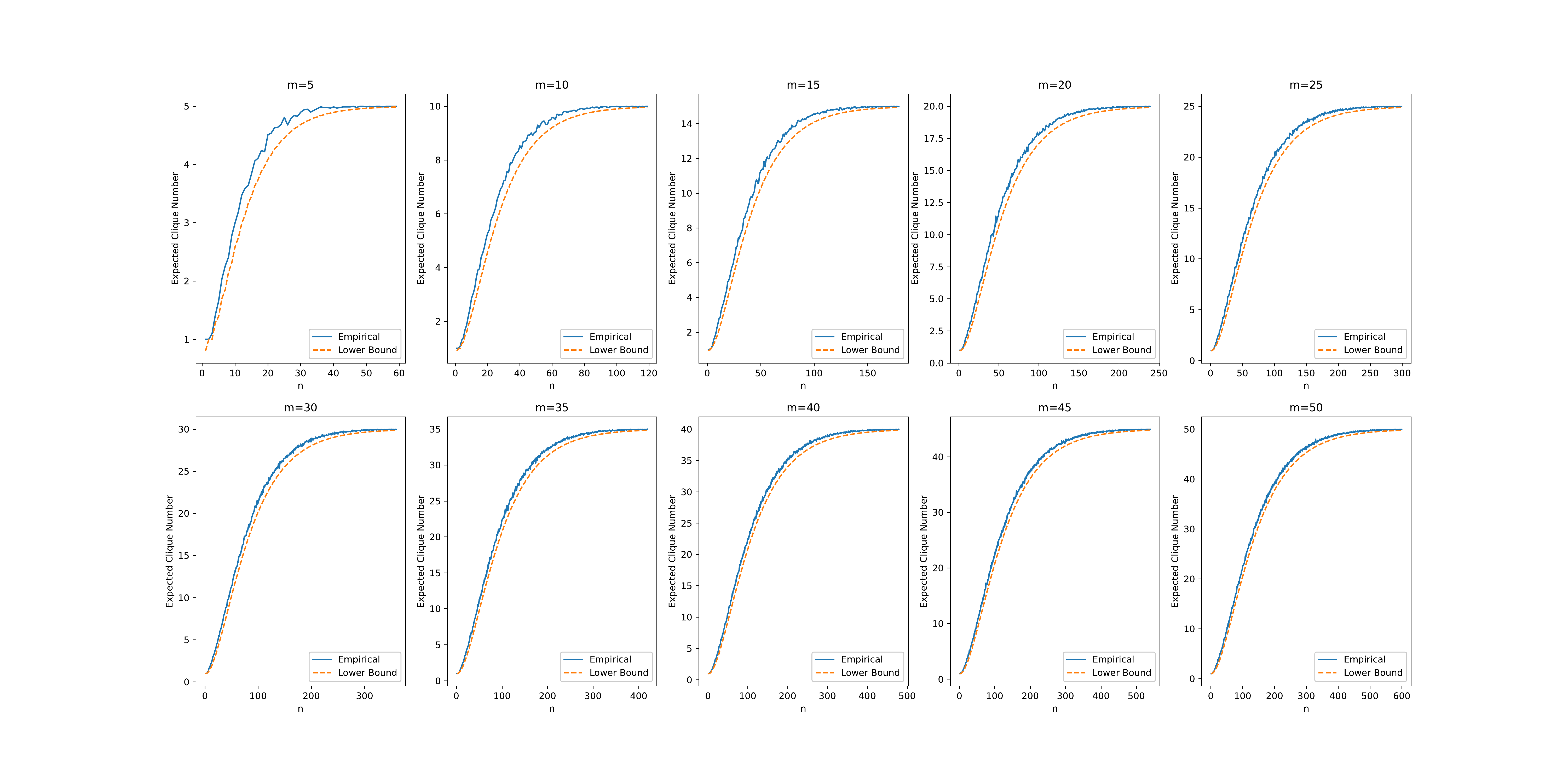}
\caption{Expected clique number of $\mathcal{N}_n$ with uniform coloring as a function of $n$.}
\label{fig:clique_number}
\end{figure}

\begin{figure}[hbt]
\centering
\includegraphics[width=0.95\textwidth]{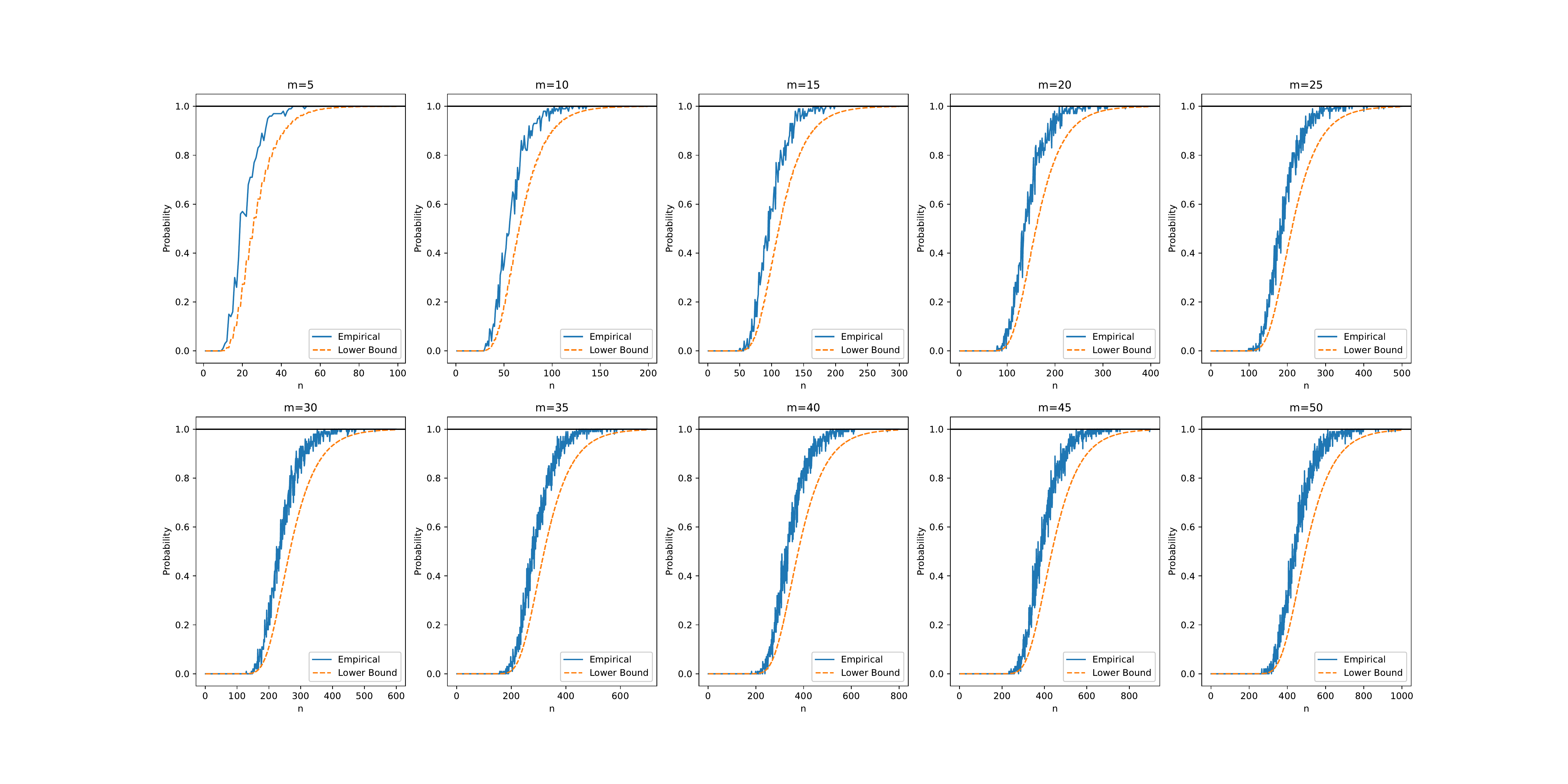}
\caption{Probability that  $\mathcal{N}_n = \Delta_{m-1}$.}
\label{fig:complete_bounds}
\end{figure}

\begin{figure}[h]
\includegraphics[scale=.38]{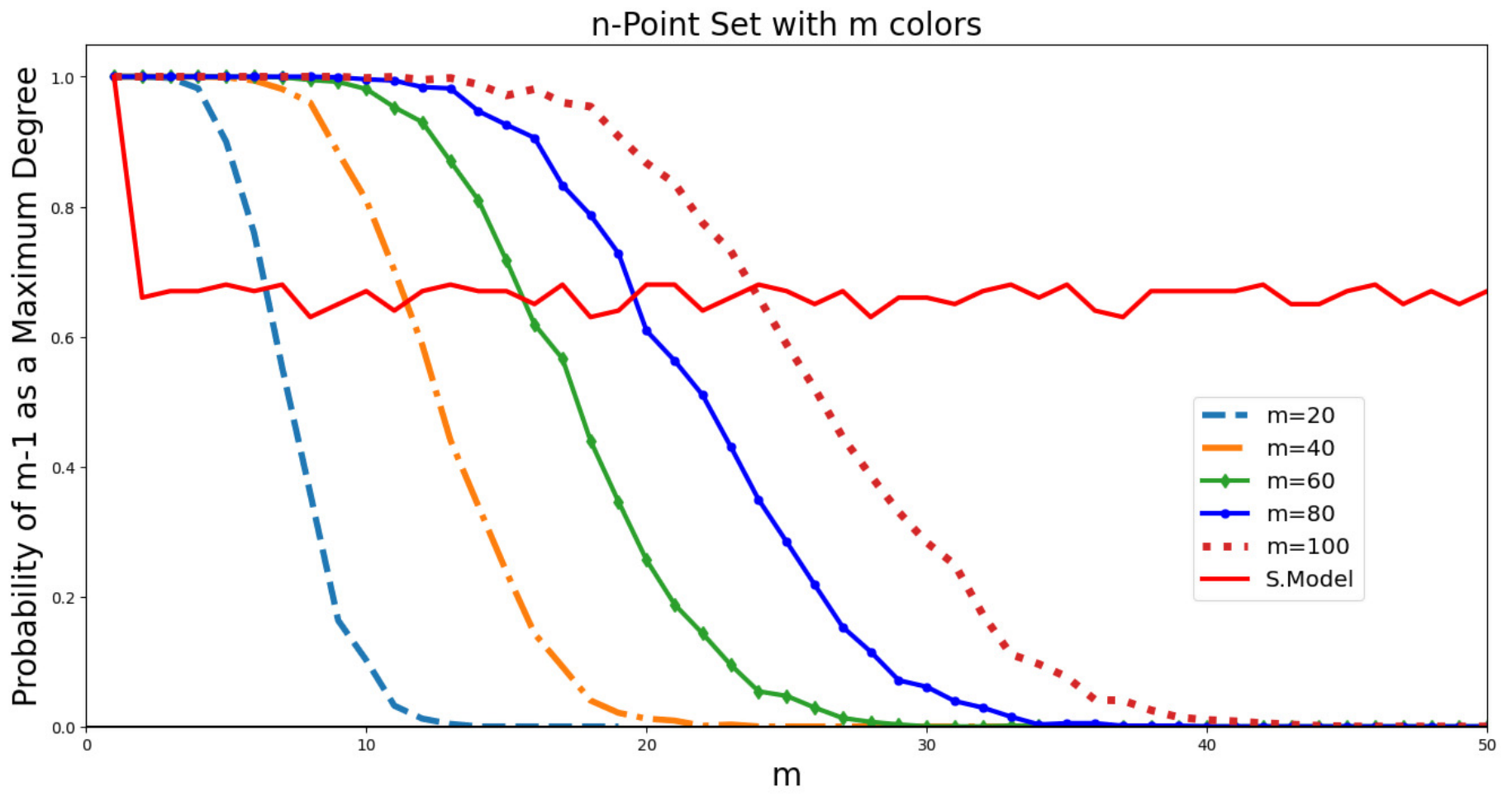}
\centering
\caption{Comparison between Scheinerman's model and ours with the probability that $\Deg(\mathcal{N}_n) = m-1$.}
\label{fig:degree}
\end{figure} 

\end{document}